\documentclass[twoside,11pt]{article}
\pdfoutput=1 
\usepackage{times, amsmath,  verbatim, enumerate, color}
\usepackage{url, graphicx, mathtools, wrapfig}
\usepackage[caption=false]{subfig}
\usepackage{hyperref}
\usepackage{epsfig, multirow, paralist, xcolor, comment, algorithm, algorithmic}
\usepackage{enumitem, lipsum}
\usepackage{jmlr2e}

\newcommand{\R}{\mathbb{R}}

\newcommand{\grad}{\nabla}

\usepackage{amsthm}
\usepackage{lipsum}
\usepackage{tikz}
\usetikzlibrary{fit,arrows,calc,positioning}
\usetikzlibrary{decorations.text}
\usetikzlibrary{shapes,arrows}
\usetikzlibrary{arrows}
\usepackage{wrapfig}
\usepackage{algorithm}
\usepackage{algorithmic}
\theoremstyle{definition} 
\theoremstyle{remark}
\newtheorem{Theorem}{Theorem}
\newtheorem{Corollary}[theorem]{Corollary}
\newtheorem{Lemma}[theorem]{Lemma}
\newtheorem{exmp}{Example}[theorem]
\usepackage{comment}
\graphicspath{{figs/}}
\usepackage{natbib}
\usepackage{url}

\hypersetup{
    colorlinks = false,
    urlcolor = {black}, urlbordercolor = {black},
    citecolor = {black},
}

\usepackage{lineno}
\renewcommand{\algorithmicrequire}{\textbf{Input:}} \renewcommand{\algorithmicensure}{\textbf{Output:}}

\newlength\myindent
\graphicspath{{figs/}}

\begin{document} 

\title{A Deterministic Nonsmooth Frank Wolfe Algorithm with Coreset Guarantees}

\author{\name Sathya N. Ravi \email ravi5@wisc.edu \\
       \addr Department of Computer Sciences\\
       University of Wisconsin Madison
       \AND
       \name Maxwell D. Collins \email mcollins@cs.wisc.edu \\
       \addr Department of Computer Sciences \\
       University of Wisconsin Madison
       \AND
       \name Vikas Singh \email vsingh@biostat.wisc.edu \\
       \addr Department of Biostatistics and 
       Department of Computer Sciences \\
       University of Wisconsin Madison\\
       \AND
       \addr Summary as a blog post: \href{https://sravi-uwmadison.github.io/2017/08/21/corensfw/}{https://sravi-uwmadison.github.io}
       }       
\editor{}

\maketitle

\begin{abstract}
  We present a new Frank-Wolfe (FW) type algorithm that
  is applicable to minimization problems with a nonsmooth convex objective.
  We provide convergence bounds and show that the scheme yields so-called {\em coreset} results for various Machine
  Learning
  problems including
  1-median, Balanced Development, Sparse PCA,
  Graph Cuts,
  and the $\ell_1$-norm-regularized Support Vector Machine (SVM) among others.
  This means that the algorithm provides approximate solutions to these problems
  in time complexity bounds that are {\em not dependent} on the size of the
  input problem. 
  Our framework, motivated by a growing body of work on {\em sublinear} algorithms for various data analysis problems, 
  is entirely deterministic and makes no use of smoothing or proximal operators. 
  Apart from these theoretical results, we 
  show experimentally
  that the algorithm is very practical and in some cases also offers significant computational advantages on large problem instances. 
  We provide an open source implementation that can be adapted for other problems that fit the 
  overall structure.
\end{abstract}


\section{Introduction}

The impact of numerical optimization on modern data  analysis has been quite significant -- today, these methods lie at the
heart of most statistical
machine learning applications in domains spanning genomics \cite{banerjee2006convex}, finance \cite{Pennanen2012} and medicine \cite{liu2009blockwise}. 
The expanding scope of these applications (and the complexity of the associated data) has continued to raise the expectations of the 
efficiency profile of the underlying algorithms which drive the analyses modules. 
For instance, the development of ``low-order'' polynomial time algorithms has long been a central focus 
of algorithmic research; the availability of a (near) linear time algorithm for a problem was considered a gold standard since 
any algorithm's runtime should, at a minimum, include the time it takes to evaluate the input data in its entirety.
But as applications that generate extremely large data sets become more prevalent, 
we encounter many practical scenarios where even a procedure that takes only linear time to process the data may be considered impractical \cite{sarlos2006improved,ClarksonW15}. Within the last decade or so, efforts 
to understand whether the standard notions of an efficient algorithm are sufficient and/or whether linear time algorithms are good enough for various turnkey applications has 
led to the study of so-called ``sublinear'' algorithms and a number of interesting results have emerged \cite{clarkson2012sublinear,rudelson2007sampling}. 

Sublinear computation is based upon the premise that a fast (albeit approximate) solution may, in some cases, be preferrable to the optimal solution 
obtained at a higher computational or financial cost. The strategy is a good fit in at least two different regimes: (a) operations on streaming 
data where an algorithm must run in real time (an inexact but 
prompt answer may suffice) \cite{BravermanC14} or (b) where the appropriate type of data (for the task) is scarce or unavailable at sample sizes necessary to utilize standard statistical tests (e.g., $\chi^2$)\cite{acharya2015optimal}. 
In either case, an important feature of most sublinear algorithms is the use of an {\em approximated version} of the decision version of the original problem.
Many of the technical results,
for specific problems/tasks have focused on deriving so-called ``property testing'' schemes as a means to establish 
what an algorithm can be expected to accomplish without reading through all of the data. 
This line of work has led to many fast sublinear algorithms for numerical linear algebra problems including matrix multiplication \cite{drineas2006fast}, computing spectral norms and leading singular vectors of the data matrix \cite{drineas2006fast} among others,  
and has been deployed in settings where the full data may have a large memory footprint. 
More recently, several authors have even shown how to extend this idea for solving convex optimization problems. 
In particular, the work by Clarkson \cite{clarkson2008-frank-wolfe} which motivates our paper obtains sublinear algorithms for a variety of problems where the feasible set is the unit simplex. 
Next, we briefly describe this framework and its relationship to {\em Coresets}, an idea often used in the design and analysis of algorithms in computational geometry.

In an important result a few years ago \cite{clarkson2008-frank-wolfe}, the authors showed that sublinear algorithms can be designed for a variety of problems using the framework of the 
well known Frank-Wolfe (FW) algorithm \cite{frank-wolfe1956}. Recall that FW algorithms can be used to solve smooth convex optimization problems when the feasible set is \emph{compact}. 
Using this framework, Clarkson unified the analysis of various problems in machine learning and statistics including regression, boosting, and density mixture estimation 
via reformulations as optimization problems with a smooth convex objective function and where the unit simplex was the feasible set. 
This setup was then shown to provide immediate results to important theoretical and practical issues including sample complexity bounds 
and faster algorithms for Support Vector Machines and approximation bounds for boosting procedures. 
This result raises an interesting question: what properties of FW algorithms enable one to design fast algorithms with provable guarantees? 
Clarkson addressed this question by drawing a contrast between FW algorithms and a (widely used) alternative strategy, i.e., projected gradient type algorithms.  
Specifically, instead of a quadratic function, in FW algorithms, we optimize a linear function over the feasible set which often yields 
a better per iteration complexity for many problems (further, the iterates always remain feasible). 
It turns out that these properties are particularly useful from a theoretical computer science/optimization point of view since 
one can frequently obtain pipelines with significantly better \emph{memory} complexity \cite{garber_linear_2016, gidel2017fwsaddle}, as we discuss next. 

An interesting consequence of the above work was a result showing an nice relationship of FW-type schemes to a concept predominantly used within computational geometry known
as \emph{coresets} \cite{agarwal2005geometric}. 
Intuitively, a coreset (typically defined in the context of a problem statement) is a subset of the input data on which an algorithm with provable
approximation guarantees for the problem at hand can be obtained which has a runtime/memory complexity that is, in some sense, independent of (or only loosely dependent on) the size of the dataset. 
 A key practical consequence of the analysis in \cite{clarkson2012sublinear} was that coresets for SVMs derived using the proposed procedure were tighter than all earlier results \cite{tsang2005core}.
For other machine learning problems which can be solved using convex optimization included in \cite{clarkson2012sublinear}, the analysis typically allows bounding
the number of nonzero elements in the decision variables (independent of the
total number of such variables).
Clearly these are very useful properties --- but we see that the setup in \cite{clarkson2012sublinear} requires that the objective function of the optimization problem be \emph{differentiable};
this is somewhat restrictive for machine learning and computer vision applications, where various nonsmooth regularizers are ubiquitous and serve
to impose structural or statistical requirements on the optimal solution \cite{bach2012optimization}. 
The overarching goal of our paper is to obtain an algorithm (very similar to the standard FW procedure) which makes no such assumption. 
Incidentally, we are also able to obtain coreset-type results for a much broader class of problems (that involve nonsmooth objective functions) --- 
providing sublinear algorithms for several of these applications and sensible heuristics for others. 
\subsection{Overview of this work and related results}

We first define the problem considered here and give an overview of some of the related works. 
Let $f$ be a convex and continuous real-valued function (but \emph{not} necessarily continuously differentiable)
over a compact convex domain $D$.
We study a specific class of finite dimensional optimization problems expressed as,
\begin{equation}
  \min_{x \in D} f(x)
  \label{eq:primal-problem}
\end{equation}
Our central result gives a new convergence bound for an algorithm  
loosely derived from a scheme first outlined 
in \cite{white1993-nondifferentiable} in the early 1990s but which seems to have been utilized
only sparingly in the literature since. 
At a high level, our deterministic method generates a sequence of iterates
$x_k$ for $k = \{0,1,2,...\}$ that provably converge in the \emph{primal--dual} gap.
Specifically, if $x_*$ is a solution to \eqref{eq:primal-problem}
we show that $f(x_k) \le f(x_*) + O(1/\sqrt{k})$,
thereby providing an \emph{a priori} bound on the suboptimality of each
iterate.
We can restate this bound, for any choice of $\epsilon > 0$ and 
\begin{equation}
  f(x_k) \le f(x_*) + \epsilon \text{ for any } k \ge K
  \label{eq:simple-bounds}
\end{equation}
for some $K \in O(1/\epsilon^2)$.

{\em Relevance.} A wide range of problems can be expressed in the form shown in
\eqref{eq:primal-problem},
and for some of these problems, bounds as in \eqref{eq:simple-bounds}
will also yield a  coreset result.
As briefly described in the previous section, within a coreset-based algorithm,
we can perform training on a very large numbers of examples
while requiring computational resources that depend on the intrinsic hardness
of the problem \emph{regardless of the amount of data collected}. We note that while 
a subsampling heuristic may show this behavior for some datasets (and for some objective functions), 
in general, a coreset (if it exists) will typically yield a stronger
approximation than ordinary subsampling. For example, coreset results have been used for
$k$-means and $k$-median clustering \cite{har-peled2005-smaller-kmeans},
subspace approximation \cite{feldman2010-subspace}, support vector machines and its variants \cite{tsang2005core, streaming_svm_coreset_2014, har2007maximum}, and robotics \cite{feldman2012-sensor-networks, balcan2013distributed}. 
For several problems we discuss here, convergence bounds on the optimization model 
show that a solution $x_K$ can be found with no more than
$K$ nonzero entries,
and these nonzeros correspond to the choice of a coreset as mentioned above.
Another consequence of such a result is that the coreset bound will be deterministic,
giving a \emph{tight} bound on the actual approximation error,
as opposed to a bound on the \emph{expected} approximation error. For some problems, our results also translate well to practice, yielding fast implementations discussed in 
the experimental section.
{\em Relation to existing work.} A small set of results pertaining to nonsmooth Frank Wolfe algorithms (and our proposed ideas) have been reported in the literature 
in the last few years. At the high level,
these approaches correspond to three different flavors: 
(1) noisy gradients: \cite{jaggi2013-frank-wolfe} considers the case where the
FW method uses an inexact gradient such that
the linear subproblems provide a solution to the exact problem within
a bounded error but assume that the objective function is \emph{smooth}; (2)  dualizing: 
\cite{jaggi2011-thesis} presented a FW dual for nonsmooth
functions employing the \emph{subgradients} of the objective.
This can be used to generate a ``certificate'' of optimality of an approximate
solution via the primal-dual gap but the algorithms do not deal with the nonsmoothness directly. In particular, for the applications studied in \cite{jaggi2011-thesis}, the nonsmooth parts in the objective function are dualized instead and treated as constraints. This approach becomes inefficient since the subproblems are complicated for many other important applications that are considered here; (3) smoothing: \cite{hazan2012-random-smoothing} presents an algorithm
for minimization of nonsmooth functions with bounded regret that applies
the FW algorithm to a randomly smoothed objective, instead
providing probabilistic bounds.  This approach closely addresses the setting that we consider, however, it has a practical bottleneck.
Observe that (at each iteration) one has to make many queries to the zeroth order oracle of the objective function in order to compute the
gradient of the smoothed (approximate) function which is often expensive when the function depends on the number of data points (may be quite large in many problems).
We do observe this behavior in our experimental evaluations in section \ref{sec:exps}. We will point out other related works in the relevant sections below. 
Overall, we see that much of the existing literature has approached this problem either via an implicit \emph{randomized smoothing} \cite{hazan2012-random-smoothing} or 
by using proximal functions \cite{argyriou2014hybrid, pierucci2014smoothing} to yield suitable gradients and provide convergence results 
for the {\em smoothed} objective. 

{\em Summary of our contributions.}  
 In section \ref{sec:prelim}, we discuss why simply using a subgradient is insufficient to prove convergence. Then, we
  introduce the basic concepts that are used in our algorithm and analysis.
  The starting point of our development is a specific algorithm mentioned in
  \cite{white1993-nondifferentiable}, 
  which is of the same general form as
  Algorithm \ref{alg:nonsmooth-stepsizek}
  of our paper in Section \ref{sec:algorithm};  
  \cite{white1993-nondifferentiable} includes a result equivalent to the ``Approximate Weak Duality''
  that we describe later. 
  In section \ref{sec:algorithm}, we derive an \emph{a priori} convergence result for a nonsmooth
generalization of the FW method, which also yields a \emph{deterministic} bound on the approximation error
and the size of the constructed coreset.  We use a much more general construction for the approximate
subdifferential $T(x, \epsilon)$ and show how this can yield novel \emph{coreset} results analogous to those shown
for the FW method in the smooth case in \cite{clarkson2008-frank-wolfe}.
Using these results, we analyze many important problems in statistics and machine learning in section \ref{sec:applications}.
While the algorithm we present may not be a silver bullet for arbitrary
nonsmooth problems (which may have {\em specialized} algorithms), 
in several settings, the results do have practical import, these are discussed in section \ref{sec:exps}. For instance, 
we show an example where our algorithm enables solving very large problem 
instances on a single desktop where fairly recent papers have deployed distributed 
optimization schemes on a cluster. 
On the theoretical side, a useful result is that our scheme can produce coresets with hard approximation
bounds.
\section{Preliminary Concepts}
\label{sec:prelim}
To introduce our algorithm and the corresponding convergence analysis,
we start with some basic definitions first.
Recall that a fundamental tool for optimization of nonsmooth functions is the
\emph{subdifferential}.
Unfortunately, utilizing just the subdifferential, turns out to be insufficient to produce the necessary convergence bounds for our analysis. 
For instance, in a very recent  paper \cite{nesterov2015complexity_cg}, the authors constructively showed that simply replacing the gradient by a subgradient is
incapable of ensuring convergence --- a simple two dimensional example demonstrated
that the FW algorithm does {\em not} converge to the optimal solution for {\em any} number of iterations.
To our knowledge, there are no clear strategies to fix this limitation at this time. 
As mentioned above, one often uses smoothing techniques as a workaround, where an auxiliary function is constructed 
which is optimized instead \cite{nesterov2005smooth, pierucci2014smoothing}. While this approach has had significant practical impact,
it involves the selection of a \emph{proximal} function. 
But there are no general recipes to choose the proximal function for a given model and is often designed based on the specific 
problem at hand. 

We observe that the convergence analysis of FW type methods  rely on the boundedness of an important quantity called the ``Curvature" Constant, 

{\small\begin{equation}
  C_f^{\text{exact}}:=\sup_{\begin{array}{c} x,s \in D \\ \alpha \in [0,1]  \end{array}}\left( \min_{d \in \partial f(x)} \frac{1}{\alpha^2} \left(
    f(y) - f(x) - \langle y - x, d \rangle
  \right) \right)
  \label{eq:cf-def-exact}
\end{equation}}
where  $y=x+\alpha(s-x)$. It relates how well the first-order information from the approximate
subdifferential {\em globally} describes the function (when $f$ is smooth, $d$ is simply the gradient of $f$).
It is this quantity that becomes unbounded even for simple nonsmooth functions making the subsequent analysis  difficult. 
To make this point clear, we now show a simple example of a nonsmooth function for which this quantity is not bounded, though in this `easy' case we can prove convergence with more
direct means.

\begin{exmp}
 Let $f(x) = |x|$ over $D = [-1,1]$.
For any $x \in (0, \frac{1}{2})$, let $s = -1 + x$ and $\alpha = 2x$.
Accordingly,
\begin{equation}
  y = x + \alpha (s - x) = x + 2x (-1 + x - x) = -x
\end{equation}
By definition,
\begin{flalign}
  C_f^\mathrm{exact}
  &\ge \min_{d \in \partial f(x)} \frac{1}{\alpha^2} \left( f(y) - f(x) - \langle y - x, d \rangle \right) \\
  &= \min_{d \in \partial f(x)} \frac{1}{4x^2} \left(f(-x) - f(x) + \langle 2x, d \rangle\right)
  \intertext{Because $f$ is differentiable at $x$, $\partial f(x) = \{1\}$,}
  &= \frac{1}{4x^2} (|-x| - |x| + 2x) = \frac{1}{2x}
\end{flalign}
Hence we have that,
\begin{equation}
  \lim_{x \to 0^+} \frac{1}{2x} = +\infty.
\end{equation}
We cannot obtain a $C_f^\text{exact}$ that can upper bound this quantity for all $x \in (0, \frac{1}{2})$.
This example shows that linear approximations based on ordinary subgradients can be arbitrarily
poor in the neighborhood of nondifferentiable points.
\end{exmp}

{\em Basic idea.} Intuitively, the subdifferential is a  discontinuous multifunction \cite{robinson2012-convex} and provides an incomplete description of the curvature $C_f$ of the 
function in the neighborhood of the nonsmooth points. After making a few technical adjustments, it turns out that we can work around this issue by making use of
\emph{approximate} subdifferentials. We deal with two constructions that yield
approximate subdifferentials.
Using these definitions, we will give the formal statement
of our convergence theorem in the next section.
The first is the \emph{$\epsilon$-subdifferential}
\cite{hiriart-urruty1993-convex},
which is defined at each point $x$ in the domain of a function $f$ as
\begin{equation}
  \partial_\epsilon f(x) := \{
    d \; | \; f(y) \ge f(x) + \langle s, y - x \rangle - \epsilon
    \text{ for all } y
  \}.
\end{equation}
Notice that the exact subdifferential is recovered when $\epsilon$ is zero,
$\partial_0 f(x) = \partial f(x)$.

While the $\epsilon$-subdifferential indeed provides the theoretical
properties for our convergence bounds,
in practical cases, it requires us to work with carefully chosen problem-specific
subsets of $\partial_\epsilon f(x)$.
A successive construction which we present later produces approximate subdifferentials that have
nicer computational properties while preserving the convergence bounds.
In this approximation, we must take the subdifferentials over a neighborhood; this idea is summarized next.

Let $N$ be an arbitrary mapping from $x$ to neighborhoods around $x$
such that $N(x, \epsilon) \to \{x\}$ as $\epsilon \to 0$.
Formally, we assume that $x \in N(x, \epsilon)$
and there exists some constant $R$ such that
$N(x, \epsilon) \subseteq \mathcal{B}(x, R \epsilon)$
for all $\epsilon > 0$.
The notation $\mathcal{B}(x, r)$ is the open ball around $x$ of radius $r$.
Assume w.l.o.g.~that $R = 1$.
Let $T$ be
\begin{equation}
  T(x, \epsilon) := \begin{cases}
    \{ \nabla f(x) \} & \text{if $f$ is differentiable on $N(x, \epsilon)$} \\
    \underset{u \in N(x, \epsilon)}{\bigcup} \partial f(u) & \text{otherwise}
  \end{cases}
  \label{eq:nbrgrad-Tdef}
\end{equation}
Here $T$ provides a set of approximate subgradients of $x$.
Indeed, if $f$ is $L$-Lipschitz, we have
$T\left(x, \frac{\epsilon}{2L} \right) \subseteq \bigcup_{u \in \mathcal{B}\left(x, \frac{\epsilon}{2L} \right)} \partial f(u) \subseteq \partial_{\epsilon} f(x)$
by Theorem 8.4.4 of \cite{robinson2012-convex}. 
It turns out that with these constructions, we will shortly obtain a {\em deterministic} non-smooth FW algorithm with accompanying \emph{convergence} results. 

\section{Convergence Results}
\label{sec:algorithm}
Using the foregoing concepts, we can adapt the procedure
in \cite{white1993-nondifferentiable} with a few technical modifications as
shown in Algorithm \ref{alg:nonsmooth-stepsizek}.
\begin{algorithm}
  \begin{algorithmic}
    \STATE Pick an arbitrary starting point $x_0 \in D$, where $D$ is the compact convex feasible set.
    \FOR{$k = 0, 1, 2, ... $}
    \STATE Let $\alpha_k := \frac{2}{k+2}$ and $\epsilon_k := \sqrt{\alpha_k}$
    \STATE Compute $s \in D$ such that
               {\footnotesize\begin{align}    
      \label{eq:alg-s-selection}
      {\max_{d \in T(x, \epsilon_k)} \langle s - x, d \rangle
        \le \min_{z \in D} \max_{d \in T(x, \epsilon_k)} \langle z - x, d \rangle}
    \end{align}}
    \STATE Update $x_{k+1} := x_k + \alpha (s - x_k)$
    \ENDFOR
  \end{algorithmic}

   \caption{ \label{alg:nonsmooth-stepsizek}
     Conditional Subgradient (based on ideas introduced in \cite{white1993-nondifferentiable})
  }
\end{algorithm}

To simplify the presentation, we will assume that the subproblems \eqref{eq:alg-s-selection} are efficiently solvable both in theory and in practice. 
In general, this is true when $T(x,\epsilon)$ is a polyhedron since 
any point $d\in T(x,\epsilon)$ can be written as a convex combination of extreme points and extreme rays. 
We also show shortly that for various statistical machine learning problems, we can solve \eqref{eq:alg-s-selection} far more efficiently even though $T(x,\epsilon)$ is not 
a polyhedron making the overall algorithm attractive in any case.
With this assumption, we may generalize the curvature constant  which plays a key role in our convergence analysis,
\begin{align}
  C_f(\epsilon) :=
  \hspace{-1.2em}
  \sup_{\substack{ x,s \in D \\ \alpha \in [0,1] \\ y = x + \alpha (s-x) }}
  \hspace{-0.6em}
   \min_{d \in T(x, \epsilon)} \frac{1}{\alpha^2} \left(
    f(y) - f(x) - \langle y - x, d \rangle
  \right)
  \label{eq:cf-def}
\end{align}
Importantly, we use approximate subgradients instead of a gradient
as in \cite{clarkson2008-frank-wolfe},
where the value of $C_f$ depends on $\epsilon$. 
We can choose an $\epsilon$ varying with iterations in such a
way as to guarantee convergence although in 
practical problems, we see that $C_f \in O(1/\epsilon)$.
In any case, note that the complexity of the algorithm does {\em not} depend on the input data, that is, we can choose $\epsilon$ to be very small such that 
$T(x,\epsilon)$ is tractable in practice. In particular, when the decision variables are defined over examples, and the subdifferential can be computed with {\em only} 
the nonzero coordinates of the decision variables, our algorithm ends up being {\em sublinear}. 
We will see this in detail in the later sections for specific problems that are also empirically verified.
Our central convergence theorem, stated in terms of this constant,
appears next.
 
\begin{Theorem}
  Suppose $f$ is $L$-Lipschitz and
  $C_f(\epsilon) \le \frac{D_f}{\epsilon}$ for some constant
  $D_f$ for any $\epsilon \le 1$.
  Then Alg. \ref{alg:nonsmooth-stepsizek}
  generates a sequence of iterates $x_k$,
  such that for the $k^\text{th}$ iteration:
  \begin{align}
    f( x_k ) - f(x_*) \le \frac{2^{5/2} L + 2^{3/2} D_f}{\sqrt{k + 2}}
    \label{eq:approx-bound}
  \end{align}
  where $x_*$ is the optimal solution to the problem in \eqref{eq:primal-problem}.
  \label{thm:convergence}
\end{Theorem}
{\em Remarks.} As the iteration count $k$ increases, the denominator in this upper bound
increases.  Since the numerator consists of constants depending
only on the definition of $f$ and $D$, this will approach $0$ as $k \to \infty$. It is important to notice 
that while this theorem provides a general convergence
bound (which is useful), it does {\em not} alone prove that Algorithm \ref{alg:nonsmooth-stepsizek}
will produce a \emph{coreset} and consequently, will be sublinear. Indeed, it is a known result for {\em smooth functions}; 
shortly, in example cases covered in Section \ref{sec:applications},
we show {\em the other key component} of the proposed algorithm:
a bound on the number of nonzeros for any solution $s$ of the subproblems
\eqref{eq:alg-s-selection}.
The key results here will depend on $D$ and $f$,
and we will seek to show that the number of nonzero entries in $s$ will
be $O(1)$ w.r.t.~the overall size of the problem.
This in turn guarantees that $x_k$ will have no more than $O(k)$ nonzeros.
Combined with the convergence bound above, this in turn will demonstrate that a sparse
$\epsilon$-approximate solution can be found with $O(1/\epsilon^2)$
nonzeros which does {\em not} depend on the problem size (e.g., number of samples/examples or dimensions).
When $x$ is a vector of the \emph{examples} in a machine learning problem,
the union of nonzeros for each $s$ found by Algorithm \ref{alg:nonsmooth-stepsizek}
will constitute a \emph{coreset}.
Theorem \eqref{thm:convergence} can be used as a \emph{general framework} that can be extended to show a coreset result
for any nonsmooth problem for which a sparsity bound on $s$ can in turn be
shown, yielding a valuable tool for sublinear algorithm design. 

\emph{Ingredients:} The proof of Theorem \ref{thm:convergence} relies on a bound in the improvement in the
objective at each iteration in terms of $C_f$ and the duality gap
between the primal objective at $x$ and a nonsmooth modification of the
Wolfe dual.
We define this dual $\omega$ of $f$ as:
\begin{align}
  \omega(x, \epsilon) :=
  \min_{z \in D} \max_{d \in \partial_\epsilon f(x)} f(x) + \langle z - x, d \rangle.
  \label{eq:eps-subdiff-dual}
\end{align}
This dual gives us a property we call \emph{approximate weak duality}.
Up to $\epsilon$, this dual is at all points less than the minimum value
of the primal objective at all points. 
\begin{Lemma}[Approx. Weak Duality]
  $\omega(x, \epsilon) \le f(y) + \epsilon$
  for all $x,y \in D$,
  \label{lem:eps-subdiff-duality}
\begin{proof}
    Take any $x,y \in D$.
  Due to the minimization in \eqref{eq:eps-subdiff-dual},
  \begin{flalign}
    \omega(x, \epsilon)
    &\le \max_{d \in \partial_\epsilon f(x)} f(x) + \langle y - x, d \rangle
  \end{flalign}
  By the definition of the $\epsilon$-subdifferential,
  for any $d \in \partial_\epsilon f(x)$,
  including the one chosen by the maximization,
  \begin{align}
    f(x) + \langle y - x, d \rangle \le f(y) + \epsilon
  \end{align}
  Yielding the lemma statement.

\end{proof}
\end{Lemma}

Denote the primal-dual gap at $x$ by
\begin{align}
  g(x, \epsilon) := f(x) - \omega(x, \epsilon)
  = \max_{y \in D} \min_{d \in \partial_\epsilon f(x)} \langle x - y, d \rangle
  \label{eq:duality-gap}
\end{align}
By Lemma \ref{lem:eps-subdiff-duality},
\begin{align}
  g(x, \epsilon) \ge f(x) - f(x_*) - \epsilon \ge -\epsilon
\end{align}
The combination of the primal-dual gap and the curvature constant $C_f$
produces a stepwise bound on the objective.
This is then used to show the a priori convergence result in
Theorem \ref{thm:convergence}.
We see that at each step of Algorithm \ref{alg:nonsmooth-stepsizek},
the objective at $x_{k+1}$ improves upon the objective at the previous
iterate $x_k$ by a term proportional to the primal-dual gap,
up to a ``curvature error'' term with $C_f$.
This is stated in the following result:
\begin{Theorem}
  For any step $x_{k+1} := x_k + \alpha (s - x_k)$,
  with arbitrary step size $\alpha \in [0,1]$ and $s \in D$ chosen 
  as in \eqref{eq:alg-s-selection}, it holds that:
  \begin{align}
    f(x_{k+1}) \le f(x_k) - \alpha g(x_k, \epsilon') + \alpha^2 C_f(\epsilon_k),
    \label{eq:stepwise-bound}
  \end{align}
  where $\epsilon' = 2 L \epsilon_k$ when $f$ is $L$-Lipschitz.
  \label{thm:stepwise-bound}
\end{Theorem}

\begin{proof}
 Write $x := x_k, y := x_{k+1} = x + \alpha (s - x)$,
  and $\epsilon := \epsilon_k$.
  From the definition of $C_f$,
  \begin{align*}
    f(y) &\le f(x) + \left( \max_{d \in T(x, \epsilon)} \langle y - x, d \rangle \right) + \alpha^2 C_f(\epsilon) \\
    &= f(x) + \alpha \left( \max_{d \in T(x, \epsilon)} \langle s - x, d \rangle \right) + \alpha^2 C_f(\epsilon) \\
    \intertext{By choice of $s$, using \eqref{eq:alg-s-selection},}
    &\le f(x) + \alpha \left( \min_{y \in D} \max_{d \in T(x, \epsilon)} \langle y - x, d \rangle \right) + \alpha^2 C_f(\epsilon)  \\
    \intertext{now using the fact that $T\left(x,\frac{\epsilon}{2L}\right)\subseteq \partial_{\epsilon}f(x)$, } 
    &\le f(x) +  \alpha \left( \min_{y \in D} \max_{d \in \partial_{\epsilon'} f(x)} \langle y - x, d \rangle \right) + \alpha^2 C_f(\epsilon) \\
    &= f(x) - \alpha \left( \max_{y \in D} \min_{d \in \partial_{\epsilon'} f(x)} \langle x - y, d \rangle \right) + \alpha^2 C_f(\epsilon) \\
    \intertext{Substituting \eqref{eq:duality-gap},}
    &= f(x) - \alpha g(x, \epsilon') + \alpha^2 C_f(\epsilon)
  \end{align*}
\end{proof}
\begin{proof}[Proof of Theorem \ref{thm:convergence}]
  Finally, we prove the main theorem by inductively applying this
  stepwise bound to find the \emph{a priori} bound in terms of $k$.

  Let $h(x) = f(x) - f(x_*)$.
  Using Theorem \ref{thm:stepwise-bound}, we first restate the
  stepwise bound as follows:
  \begin{flalign*}
    h(x_{k+1})
    &\le h(x_k) - \alpha_k g(x_k, 2 L \epsilon_k) + \alpha_k^2 C_f(\epsilon_k) \\
    \intertext{By Lemma \ref{lem:eps-subdiff-duality},}
    &\le h(x_k) - \alpha_k (h(x_k) - 2 L \epsilon_k) + \alpha_k^2 C_f(\epsilon_k) \\
    &\le (1 - \alpha_k) h(x_k) + 2L\alpha_k \epsilon_k + \frac{\alpha_k^2}{\epsilon_k} D_f \tag{{\large$\ast$}}
    \label{eq:proof-step-ae-bounds}
    \intertext{Now, if we substitute the constant $E := 2L + D_f$,
      and the choice $\epsilon_k = \sqrt{\alpha_k}$ as in the algorithm,}
    &= (1 - \alpha_k) h(x_k) + \alpha_k^{3/2} E
  \end{flalign*}
  We will  see in the experiments how this choice of $\epsilon_k$ makes the algorithm practically useful. Using this notation to rewrite the theorem, we seek to show
  $h(x_k) \le \frac{2^{3/2} E}{\sqrt{k+2}} = 2 E \sqrt{\alpha_k}$.

  We prove this by induction.
  The base case comes from applying the stepwise bound with $k=0$ and
  $\alpha_k = 1$
  to get $h(x_1) \le E \le \frac{2^{3/2} E}{\sqrt{3}}$.
  For $k \ge 1$, making the induction step,
  \begin{flalign*}
    h(x_{k+1})
    &\le \left(1 - \alpha_k \right) (2 E \sqrt{\alpha_k}) + \alpha_k^{3/2} E \\
    &= \left(1 - \frac{2}{k+2} \right) \frac{2^{3/2} E}{\sqrt{k+2}}
    + \frac{2^{3/2}}{(k+2)^{3/2}} E \\
    &= 2^{3/2} E \sqrt{k+2} \left( \frac{1}{k+2} - \frac{1}{(k+2)^2} \right)
  \end{flalign*}
  Now note that
  \begin{flalign*}
    \left( \frac{1}{k+2} - \frac{1}{(k+2)^2} \right)
    &= \frac{1}{k+2} \frac{k+2-1}{k+2} \\
    &\le \frac{1}{k+2} \frac{k+2}{k+3} = \frac{1}{k+3}
  \end{flalign*}
  So
  \begin{align*}
    h(x_{k+1}) \le 2^{3/2} E \frac{\sqrt{k+2}}{k+3}
    \le 2^{3/2} E \frac{\sqrt{k+3}}{k+3}
    = \frac{2^{3/2} E}{\sqrt{k+3}}
  \end{align*}
\end{proof}
\emph{Comments:} 
A key intuition about the algorithm can be gained from the proof
step in \eqref{eq:proof-step-ae-bounds}.
This inequality implies that the improvement in the objective for each step
has a bounded ``error'' with two parts.
First is an error linear in $\alpha_k$ that comes from our use
of an approximate subgradient rather than an exact subgradient.
This error is proportional to $\epsilon_k$.
Second, we have a ``curvature'' error that arises from the nonlinearity
of our objective function and is instead \emph{inversely} proportional to 
$\epsilon_k$.
By choosing a decaying $\epsilon_k$ such that
$\frac{1}{\epsilon_k} \in \Theta\left(\frac{1}{\sqrt{\alpha_k}}\right)$
we send both error terms to $0$ as $k \to \infty$.
\section{Approximable Nonsmooth Problems}
\label{sec:applications}

In this section, we  demonstrate how the foregoing ideas can be applied and made more specific, on a case by case basis, for various problems arising 
in statistical machine learning. 
Since many of the steps and quantities in the algorithm and theorems described in the previous section 
depend on details of the individual problem,
we use these examples to show the potential value of our results.
In this section, each problem we present yields a choice of objective function $f$ and feasible region $D$.
For these choices, we then describe the approximate subdifferentials
and the resulting subproblems to determine the step direction.
We show that the antecedents of Theorem \ref{thm:convergence} are satisfied for a number of machine learning problems
and provide $O(1/\epsilon)$ bounds for $C_f(\epsilon)$.
A problem-specific bound on $C_f$ is the final component of establishing
an iteration bound for that problem.
It guarantees that our scheme for selecting $\epsilon_k$ will provide
convergence,
and that we have finite values for the numerator in \eqref{eq:approx-bound}. Therefore, 
the quotient will approach $0$ in the limit. 

{\em Overview of problems.} We first analyze the supremum of linear functions and  compute the $C_f$. With this in hand, we will present a formulation of 
$\ell_1$ SVM that is suitable for our analysis. In particular, we show $C_f$ bounds using the result from piecewise linear functions and then 
proceed to discuss how the subproblems \eqref{eq:alg-s-selection} can be efficiently solved. We discuss the tractability and provide tight coreset bounds for this problem. 
Later, we analyze the multiway graph cuts model by adapting the techniques used in the $\ell_1$ SVM problem and finally, we also describe how the analysis 
extends to the 1-median problem, Sparse PCA and Balanced development.
\subsection{Supremum of Linear Functions}
\label{sec:piecewise-linear}

We can define a piecewise linear convex function on $\R^n$ by
\begin{align}
  f(x) := \max_{i=1,...,p} a_i^T x + b_i.
  \label{eq:piecewise-linear-def}
\end{align}
If we let $A \in \R^{p \times n}$ be the matrix such that row $i$
is equal to $a_i$, and $b \in \R^p$ be the vector that has elements $b_i$.
We know that $f$ is Lipschitz-smooth with constant equal to the operator
norm $\|A\|$ and
$f(x) \equiv \hat{f}(Ax + b)$ for

\begin{align}
  \hat{f}(\hat{x}) = \max_{i = 1,...p} \hat{x}_i
\end{align}
so that
\begin{align}
  f(x) := \hat{f}(Ax + b)
  \label{eq:piecewise-fdef}
\end{align}
The subdifferentials of $\hat{f}$ are given by
\begin{align}
  \partial \hat{f}(\hat{x}) = \left\{
  e_i \; \middle| \; \left( \max_{j=1,...,p} \hat{x}_j \right) = \hat{x}_i
    \right\}
  \label{eq:maxelt-subdiff}
\end{align}
For the purposes of calculating approximate subdifferentials 
$\hat{T}$ and $T$ of $\hat{f}$ and $f$ respectively (defined in section \ref{sec:prelim}), we use the neighborhoods:
\begin{align}
    &\hat{N}(\hat{x}, \epsilon) =
    \{ \hat{u} \; | \; \|\hat{u} - \hat{x}\|_\infty \le \epsilon \},
    \\
    &N(x, \epsilon) = \{u \; | \; \|Au - Ax\|_\infty \le \epsilon \text{ and } u - x \in \mathrm{im}\left(A^T\right) \},
\end{align}
where $A^T$ is the transpose of $A$ and $\mathrm{im}\left(A^T\right)$ is the row space of $A$.

\begin{Theorem}  Define
\begin{align}
  V(\hat{x}, \epsilon) = \left\{e_i \; \middle| \;
  \hat{x}_i \ge \left( \max_j \hat{x}_j \right) - 2 \epsilon
  \right\}
\end{align}
and let $\mathrm{Co}S$ denote the convex hull of the set $S$. Then, the approximate subdifferential is given by
  $\hat{T}(\hat{x}, \epsilon) = \mathrm{Co} V(\hat{x}, \epsilon)$
  \label{thm:piecewise-linear-Tverts}
\end{Theorem}

\begin{proof}
  Note that both sides will be a singleton set equal to
  $\{\nabla \hat{f}(\hat{x})\}$
  iff $\hat{f}$ is differentiable on $\hat{N}(\hat{x},\epsilon)$.
  We therefore prove only the case that $\hat{f}$ is nondifferentiable
  somewhere in the neighborhood.\\

  \noindent \textbf{Forward Direction: $\hat{T}(\hat{x}, \epsilon) \subseteq \mathrm{Co} V(\hat{x}, \epsilon)$}

  Take any $\hat{u} \in \hat{N}(\hat{x}, \epsilon)$. We have, 
  \begin{align}
    \partial \hat{f}(\hat{u})
    = \mathrm{Co}\left\{ e_i \; \middle| \; \hat{u}_i = \left( \max_j \hat{u}_j \right) \right\}
    \label{eq:maxelt-u-subdiff}
  \end{align}
  Now take any $i$ such that $e_i \in \mathrm{vert}(\partial \hat{f}(\hat{u}))$ where 
  $\mathrm{vert}(S)$ denotes the vertices of a polyhedral set $S$.
  If $S$ is a finite set of linearly independent points,
  then $\mathrm{vert}(\mathrm{Co}S) = S$ 
  (which in this case is the set of basis vectors in \eqref{eq:maxelt-u-subdiff}).
  Because $\|\hat{u} - \hat{x}\|_{\infty} \le \epsilon$,
  \begin{align}
    \hat{x}_i \ge \hat{u}_i - \epsilon = 
    \left( \max_j \hat{u}_j \right) - \epsilon
    \ge \left( \max_j \hat{x}_j \right) - 2\epsilon
  \end{align}
  Therefore
  \begin{align}
    \mathrm{vert}(\partial \hat{f}(\hat{u})) &\subseteq V(\hat{x}, \epsilon) \\
    \implies\partial \hat{f}(\hat{u}) &\subseteq \mathrm{Co} V(\hat{x}, \epsilon)
  \end{align}
  Since this is true for all $u \in \hat{N}(\hat{x}, \epsilon)$, we see that
  $\hat{T}(\hat{x}, \epsilon) \subseteq \mathrm{Co} V(\hat{x}, \epsilon)$.\\

  \noindent \textbf{Reverse Direction: $\hat{T}(\hat{x}, \epsilon) \supseteq \mathrm{Co} V(\hat{x}, \epsilon)$}

  Let $\hat{u}$ be the vector defined elementwise by
  \begin{align}
    \hat{u}_i = \begin{cases}
      \left( \max_j \hat{x}_j \right) - \epsilon
      & \text{if } \hat{x}_i \ge \left( \max_j \hat{x}_j \right) - 2 \epsilon \\
      \hat{x}_i & \text{otherwise}
    \end{cases}
  \end{align}

  If we take some $i$ such that $\hat{x}_i \ge \left( \max_j \hat{x}_j \right) - 2 \epsilon$,
  then
  \begin{align}
    \hat{u}_i - \hat{x}_i &= \left( \max_j \hat{x}_j \right) - \epsilon - \hat{x}_i \\
    &\le \left( \max_j \hat{x}_j \right) - \epsilon - \left( \left( \max_j \hat{x}_j \right) - 2 \epsilon \right)  = \epsilon     
      \end{align}
      Similarly, it also holds that,
        \begin{align}
\hat{x}_i - \hat{u}_i\le \left( \max_j \hat{x}_j \right) - \hat{u}_i
    = \left( \max_j \hat{x}_j \right) - \left( \left( \max_j \hat{x}_j \right) - \epsilon \right) = \epsilon
  \end{align}
  And $\hat{u}_i = \hat{x}_i$ for the ``otherwise'' case.
  Therefore $\|\hat{u}_i - \hat{x}_i\|_\infty \le \epsilon$
  and $\hat{u}_i \in \hat{N}(\hat{x}, \epsilon)$.

  Furthermore, note that clearly $\hat{f}(\hat{u}) = \max_j \hat{x}_j - \epsilon$,
  and for any $i$ such that
  \begin{align}
    \hat{x}_i < \left( \max_j \hat{x}_j \right) - 2 \epsilon
  \end{align}
 we have that, for all $\epsilon > 0$,
  \begin{align}
    \hat{x}_i < \left( \max_j \hat{x}_j \right) - \epsilon = \hat{f}(\hat{u})
  \end{align}
  Accordingly,
  \begin{align}
    &\mathrm{vert}(\partial \hat{f}(\hat{u})) = \hat{V}(\hat{x}, \epsilon) \\
\implies    &\hat{T}(\hat{x}, \epsilon) \supseteq \partial \hat{f}(\hat{u}) = \mathrm{Co} \hat{V}(\hat{x}, \epsilon).
  \end{align}
\end{proof}

\begin{Lemma}
  $T(x, \epsilon) = A^T \hat{T}(Ax + b, \epsilon)$
  \label{lem:piecewise-linear-affineT}
\end{Lemma}

\begin{proof}
  This follows from the fact that
  $\partial f(u) = A^T \partial \hat{f}(Au + b)$,
  and $A~N(x, \epsilon) + b = \hat{N}(Ax + b, \epsilon)$.
\end{proof}

\begin{Corollary}
  $T(x, \epsilon) = A^T \mathrm{Co} V(Ax + b, \epsilon) = \mathrm{Co} A^T V(Ax + b, \epsilon)$
\end{Corollary}

\begin{proof}
  The first equality is simply the combination of Lemma \ref{lem:piecewise-linear-affineT} and Theorem \ref{thm:piecewise-linear-Tverts}.
  The second is clear from the definition of convex hull.
\end{proof}
So the approximate subdifferentials $\hat{T}$ and $T$,
 of the functions $\hat{f}$ and $f$ respectively are related by
\begin{align}
    &\hat{T}(\hat{x}, \epsilon) = \mathrm{Co} V(\hat{x}, \epsilon) \\
    \text{ and } 
    \quad &T(x, \epsilon) = A^T \mathrm{Co} V(Ax + b, \epsilon) 
  \label{eq:piecewise-T-from-vertices}
\end{align}
for an appropriate choice of neighborhoods.
\begin{Lemma}
  If $\hat{y} \in \hat{N}(\hat{x}, \epsilon)$
  then $\partial \hat{f}(\hat{y}) \subseteq \hat{T}(\hat{x}, \epsilon)$.
  \label{lem:piecewise-linear-ysubgrad}
\end{Lemma}

\begin{proof}
  If $\hat{f}$ is nondifferentiable anywhere on $\hat{N}(\hat{x}, \epsilon)$,
  the inclusion follows from the definition of $\hat{T}$. So we will consider the case when the neighborhood contains no nondifferentiable points.

  Let $i' := \arg \max_i \hat{x}_i$.
  If $\hat{f}$ is instead differentiable everywhere on the neighborhood
  around $\hat{x}$,
  this means $\hat{u}_i < \hat{u}_{i'}$
  for all $u \in \hat{N}(\hat{x}, \epsilon)$ for any $i \neq i'$.
  Therefore $\grad \hat{f}(\hat{x}) = \grad \hat{f}(\hat{y}) = \{e_{i'}\}$
  and the lemma is true as an equality.
\end{proof}
\begin{Theorem}
  Consider any $\hat{x}, \hat{s} \in \R^p$ and $\alpha \in [0,1]$.
  Let $\hat{y} := \hat{x} + \alpha (\hat{s} - \hat{x})$.
  Then
  \begin{align}
    \min_{\hat{d} \in \hat{T}(x,\epsilon)}
    \frac{1}{\alpha^2} \left(
    \hat{f}(\hat{y}) - \hat{f}(\hat{x})
    - \langle \hat{y} - \hat{x}, \hat{d} \rangle \right)
    \le \frac{2}{\epsilon} \| \hat{s} - \hat{x} \|^2_\infty.
    \label{eq:maxelt-Cf-distance}
  \end{align}
  \begin{proof}
    It is sufficient to look at the following two cases, \\
    
    \noindent \textbf{Case 1: $\hat{y} \in \hat{N}(\hat{x}, \epsilon)$}
  
  Let $i' = \arg \max_i \hat{y}_i$.
  Then $e_{i'} \in \partial \hat{f}(\hat{y}) \subseteq \hat{T}(x,\epsilon)$
  from \eqref{eq:maxelt-subdiff} and Lemma \ref{lem:piecewise-linear-ysubgrad}.
  Then the left hand side of \eqref{eq:maxelt-Cf-distance}
  is bounded above by
  \begin{align}
    \hat{f}(\hat{y}) - \hat{f}(\hat{x}) - \langle \hat{y} - \hat{x}, e_{i'} \rangle
    &= \hat{y}_{i'} - \left(\max_i \hat{x}_i\right) - (\hat{y}_{i'} - \hat{x}_{i'})\\
    &= \hat{x}_{i'} - \left(\max_i \hat{x}_i\right) \le 0
  \end{align}

  \noindent \textbf{Case 2: $\hat{y} \notin \hat{N}(\hat{x}, \epsilon)$}

  First note that in this case we have a lower bound on $\alpha$:
  \begin{align}
    \alpha \ge \frac{\epsilon}{\| \hat{s} - \hat{x} \|_\infty}
    \label{eq:piecewise-linear-alpha-lower-bound}
  \end{align}
  Then we use the fact that $\|\hat{d}\|_1 \le 1$ for any
  $\hat{d} \in \partial \hat{f}(\hat{u})$ for any $\hat{u} \in \R^p$
  to bound the left hand side of \eqref{eq:maxelt-Cf-distance} by
  \begin{align}
     &\hat{f}(\hat{y}) - \hat{f}(\hat{x}) + \|\hat{y} - \hat{x}\|_\infty
    \\&= \left( \max_i \; \hat{x}_i + \alpha (\hat{s}_i - \hat{x}_i) \right)
      - \left( \max_i \hat{x}_i \right)
    + \alpha \|\hat{s} - \hat{x}\|_\infty \\
    &\le \left( \max_i \hat{x}_i \right)
    + \alpha \left( \max_i \hat{s}_i - \hat{x}_i \right)
      - \left( \max_i \hat{x}_i \right)
    + \alpha \|\hat{s} - \hat{x}\|_\infty \\
    &\le 2 \alpha \|\hat{s} - \hat{x}\|_\infty
    \label{eq:piecewise-linear-cflem-last}
  \end{align}
  The bound in the lemma then follows from 
  \eqref{eq:piecewise-linear-alpha-lower-bound}
  and \eqref{eq:piecewise-linear-cflem-last}.
\end{proof}

\end{Theorem}

Denote the diameter of a set $S$ with respect to the $q$-norm
by $\mathrm{diam}_q(S) = \max_{x,y \in S} \|x - y\|_q$.
Then, we can further bound \eqref{eq:maxelt-Cf-distance} by the diameter
of the bounded feasible set.

\begin{Corollary}
  Given the problem of minimizing a piecewise linear function
  $f(x) = \max_i (Ax + b)_i$ over a bounded set $D$,
  the corresponding $C_f$ is bounded above by
  $ C_f(\epsilon) \le \frac{2}{\epsilon} \left( \mathrm{diam}_\infty(A~D) \right)^2$.
  \label{cor:piecewise-linear-Cfbound}
\end{Corollary}
\begin{proof}
  This comes from the definition
  \begin{align}
    \mathrm{diam}_p(S) = \max_{x,y \in S} \|x - y\|_p
  \end{align}
  so that $\|\hat{s} - \hat{x}\|_\infty \le \mathrm{diam}_\infty(A~D)$,
  for all $\hat{s} = As + b, \hat{x} = Ax + b$ with $x,s \in D$. 
\end{proof}

\subsubsection{$\ell_1$-norm-regularized SVM}
\label{sec:l1-svm}

The first Machine Learning problem we analyze is the $\ell_1$-norm-regularized
SVM.
Since the $\ell_1$ norm is sparsity-inducing,
the regularization $\|w\|_1$ gives an optimization problem that finds
a separating hyperplane in only a small subset of the features.
This is expected to perform well in problems where there is a large number of
redundant or noisy features included with a few informative features.

Suppose we are given training examples
$(x_i, y_i) \in \R^d \times \{-1,1\}$.
To train a classifier on this input we start with a hard-margin variant of the $\ell_1$-norm-regularized SVM \cite{mangasarian1999-generalized-svm}:
\begin{align}
    \min_{w\in\R^d,\gamma\in\R}  \|w\|_1  
    \text{ s.t. }   y_i(w^T x_i - \gamma) \ge 1  ~\forall~  i
  \label{eq:hardmargin-l1svm-primal}
\end{align}
Define matrix $A^+ \in \R^{d \times m}$ that has the positive examples in the
columns, and $A^- \in \R^{d \times n}$ has the negative examples in the columns.
Bennet \& Bredensteiner \cite{bennett2000-svm-geometry} provide a
dual to the \emph{soft-margin} $\ell_1$-norm-regularized SVM
with hinge loss, related to polytope distance formulations of the $\ell_2$-norm-regularized SVM \cite{gartner2009-polytope-distance}:
\begin{align}
    \min_{u\in \Delta_m,v\in \Delta_n} \quad \|A^+ u - A^- v\|_\infty ~
    \text{s.t. }   
    0\leq u, v \le \frac{1}{R}.   
  \label{eq:bennet-l1svm-dual}
\end{align}
The elementwise upper bounds on $u$ and $v$ and the parameter $R$
come from the reduced convex hull construction in that paper,
to which we refer the interested reader. 
When $R = 1$, this yields a dual to the hard-margin problem. 
In the hard-margin case, the feasible set is $D := \Delta_m \times \Delta_n$,
where $\Delta_m$ and $\Delta_n$ denotes the unit simplices in $\R^m$
and $\R^n$ respectively,
and the feasible set is a subset of $D$ in the soft-margin case.
Therefore, using Corollary \ref{cor:piecewise-linear-Cfbound}, we get
  $C_f(\epsilon) \le \frac{2}{\epsilon} \left( \mathrm{diam}_\infty (A^+ \Delta_m) + \mathrm{diam}_\infty (A^- \Delta_n) \right)$.

\paragraph{\bf Frank Wolfe Subproblems.}

For the dual problem in \eqref{eq:bennet-l1svm-dual},
express our objective as $f(x) := \|Ax\|_\infty$,
for $A$ and $x$ such that $Ax = A^+ u - A^-v$.
We can then write $f$ as in
\eqref{eq:piecewise-linear-def}, where the
linear functions are given by the rows of $A$ and the rows of $-A$.
Define the approximate subdifferential as in \eqref{eq:piecewise-T-from-vertices}, and in this case
\begin{align}
    V(\hat{x}, \epsilon) = 
    \{e_i \; | \; \hat{x}_i \ge \|\hat{x}\|_\infty - 2\epsilon \}
    \cup \{-e_i \; | \; \hat{x}_i \le -\|\hat{x}\|_\infty + 2\epsilon \}.
  \label{eq:svm-T-vertices}
\end{align}

We now turn to the subproblems in \eqref{eq:alg-s-selection}
for this  $T(x,\epsilon)$ and feasible region $D$.
Observe that the maximum over $d$ will have a solution equal
to a vertex of $T(x, \epsilon)$.
An optimal $z$ can then be found by:
\begin{align}
    \min_{z \in \R^{m + n}, ~\mu \in \R} \quad & \mu \\
    \text{s.t.} \quad
    & \langle Az - Ax, d_v \rangle \le \mu
    \quad \text{for } d_v \in V(Ax, \epsilon) \\
    & \mathbf{1}^T z_+ = 1, \mathbf{1}^T z_- = 1,0 \le z \le \frac{1}{R},  
  \label{eq_z_sub}
\end{align}

where $z_+$ and $z_-$ are the indices of the positive and negative examples
respectively.
There will be an optimal $(z, \mu)$ such that $n+m+1$ linearly independent
constraints are active, and
$z_i \le \frac{1}{R}$
will be active for no more than $2 \lceil R \rceil$ indices $i$.
This means that there is a solution $z$ with no more than
$2 \lceil R \rceil + |V(Ax, \epsilon)| - 1$ nonzeros.  Clearly, since the number of nonzero entries is strictly less than the total number of dimensions, the algorithm is sublinear.

{\em Remarks.} Here we note that the construction of coreset depends on the assumption that each iteration involves only $O(1)$ examples. This is true under the conditions which we describe now. Typically, $R$ will be small for well-separated classes. Due to complementary slackness,
for small $\epsilon$ we expect that $V(Ax, \epsilon)$ at the optimum
will correspond to the nonzero features in the sparse separating hyperplane.
This provides a central intuition behind the use of coresets.
When a problem is ``easy'' in a geometric or AI sense, having
well-separated classes with few relevant features,
the resulting coreset is small.
Conversely, in problems that are ``hard'' 
the coreset size may approach the size of the original dataset. We also note that the steps, $s_k$ may not be at a vertex of $D$, slightly complicating the coreset construction shown empirically in the later section. However, we \emph{still} see useful results for nonsmooth $f$ and polyhedral $D$ if $s$ is on a low-dimensional face of $D$, then $s$ is a convex combination of a small number of vertices of $D$. Then the step will introduce more than one “atom,” indeed all of the vertices of the lowest-dimensional face on which $s$ lies, but still a small number.

\paragraph{Tractability of subproblems}
If we look at the subproblems from a Computational Geometry perspective, our algorithm will look roughly like Gilbert's algorithm \cite{gilbert1966iterative}. This is similar to what is described in \cite{bennett2000-svm-geometry}, though the $\ell_1$-SVM case is not developed fully in that paper and the $T(x,e)$ construction introduces many additional issues. We see that $\ell_1$-SVM subproblems i.e.,  \eqref{eq_z_sub} is solved in $O(R |V(Ax, e)|)$ time, given which example points determine the axis-aligned bounding box of the $R$-reduced convex hull of each class. The bounding box can be computed in time $O(Rnd)$ and $O(Rd)$ space: a cheap one-time cost performed before the optimization. One can also construct the bounding box lazily in $O(Rnd^*)$ time for $d^*$ the cardinality of the union of all $V(Ax, e)$ across all iterations, and for problems that are easy in a geometric sense, we will have $d^* \ll d$ making the procedure sublinear. Hence it is important to see that one would \emph{not} necessarily directly perform the optimization in \eqref{eq:alg-s-selection} using a generic solver. 

\paragraph{Coreset Bound}
The above results for $\ell_1$-norm SVM provide the necessary lemmas to
show a \emph{coreset} bound for this problem.
The construction of the coreset in this work follows the same intuition as when
using the ordinary Frank Wolfe method \cite{clarkson2008-frank-wolfe}.
For each iterate $x_k \in \R^n$,
and the corresponding subproblem solutions $s_k$,
each index into the vector represents an input \emph{example}.
We may take those examples for which the index in $s_k$ is nonzero,
and call this set $S_k$.
Using Algorithm \ref{alg:nonsmooth-stepsizek} to build a coreset,
we take the coreset at iteration $K$ to be $\bar{S}_K = \bigcup_{k=1}^K S_k$.

Suppose we seek a coreset that will provide a $ \epsilon-$ approximation
solution.
By Theorem \ref{thm:convergence}, if we choose $K$ to ensure that both
sides of \eqref{eq:approx-bound} are (multiplicatively) bounded above by $(1 + \epsilon)$,
then the union $\bar{S}_K$ will provide this coreset.
This approximation bound will be satisfied for
\begin{align}
  K \ge \frac{(2^{5/2} L + 2^{3/2} D_f)^2}{(1 + \epsilon)^2} - 2.
  \label{eq:K-bound}
\end{align}
The set of examples $\bar{S}_K$ will then be a $(1+\epsilon)$ coreset
for any $K$ satisfying this bound.
We know from the above that the $\ell_1$-SVM dual objective will be $\|A\|-$Lipschitz,
and that $C_f(\epsilon) \le \frac{D_f}{\epsilon}$
for $D_f = 2 \left( \mathrm{diam}_\infty(A^+ \Delta_m) + \mathrm{diam}_\infty(A^- \Delta_n) \right)$.

For a coreset to be empirically useful, we also want it to be of small cardinality.
The next step is therefore to bound $|\bar{S}_K|$.
We earlier showed that $|S_k| \le 2 \lceil R \rceil + d^* - 1$ for all $k$.
Using a union bound,
\[
  |\bar{S}_K| \le K \left( 2 \lceil R \rceil + d^* - 1 \right).
\]
Combining these, we can use Algorithm \ref{alg:nonsmooth-stepsizek}
to construct a $\epsilon$ coreset of size $O\left( \frac{1}{\epsilon^2} \right)$. 
\subsubsection{Balanced development}
\label{sec:balanced-devel}


A direct application of the piecewise linear discussion in
Section 4.1 
is the problem of \emph{balanced development},
see \cite{Nesterov05primal-dualsubgradient}.
Let $A=(a_1,...,a_N)\in\mathbb{R}^{m\times n}$ where $m$ is the number of attributes and $n$ is the number of products. Hence, the entry $A_{ij}$ is the \textit{concentration} of attribute $i $ in product $ j$.
Let $b_i$ be the minimum concentration of attribute $i$ that we require
in our solution.
Also, $p_j\geq 0$ is the unitary price of product $j$.
The problem is then to make an investment decision between the products
to maximize the minimum amount of each attribute.
In \cite{Nesterov05primal-dualsubgradient}, the author provides a dual problem:
\begin{align}
  \tau^* = \min_y\{\phi(y):y\in\Delta_m\}
  \label{eq:balacned-development-primal}
  \text{ where }
  \phi(y)=\max_{1\leq j\leq n}\frac{1}{p_j}\langle a_j,Ey\rangle.
\end{align}
$E \in \R^{m \times m}$ is diagonal with $E_{ii} =\frac{1}{b_i}$.
Corollary \ref{cor:piecewise-linear-Cfbound} implies
$C_f(\epsilon) \le \frac{2}{\epsilon} (\mathrm{diam}_\infty(A\Delta_m))^2$,
and combined with the simplicial feasible region this provides a coreset over the attributes.

\subsection{Multiway Graph Cuts}
\label{sec:graph-cuts}
We consider the model for multi-label graph cuts in \cite{calinescu1998-multiway-cut,niu2011-hogwild}.
This problem seeks to assign one of $d$ labels to each of $n$ nodes.
The graph structure is determined by similarity matrix $W$,
with weight $w_{uv}$ between vertices $u$ and $v$.
A convex relaxation of this problem is given by:
\begin{align}
    \min_x \quad  \sum_{u \sim v} w_{uv} \| x_u - x_v \|_1 \quad
    \text{s.t.} \quad 
    x_v \in \Delta_d \text{ for } v = 1,...,n.
\end{align}
We further constrain a set of ``seed'' nodes to have $x_v = e_i$
if the label of seed $v$ is $i \in \{1,...,d\}$.
Let $X \in \R^{d \times n}$ be the decision variable,
a matrix such that each column is the soft labeling for a node.
If we let $B$ be the incidence matrix for an orientation of the graph,
and $I$ the $d \times d$ identity matrix, then we can rewrite the objective as:
\begin{align}
  f(\mathrm{vec}(X)) =
  \|(B \otimes I) \mathrm{vec}(X)\|_1 = \|\mathrm{vec}(X B^T)\|_1,
\end{align}
where $\otimes$ is the matrix Kronecker product.

Now define $\hat{f}(\hat{x}) := \|\hat{x}\|_1$ so that
$f(\mathrm{vec}(X)) = \hat{f}((B \otimes I) \mathrm{vec}(X))$.
We start by deriving the subdifferentials 
for $\hat{f}$.
For the approximate subdifferentials of $\hat{f}$ we choose the neighborhoods
$\hat{N}(\hat{x}, \epsilon)$ to be the $\ell_1$- ball of radius $\epsilon$ at $x$.
\begin{Lemma} Similar to lemma \ref{lem:piecewise-linear-ysubgrad}, the subdifferential of the neighborhood is a subset of the approximate differential, that is, if $\hat{y} \in \hat{N}(\hat{x}, \epsilon)$, then
  $\partial \hat{f}(\hat{y}) \subseteq \hat{T}(\hat{x}, \epsilon)$.
\end{Lemma}

\begin{proof}
  If $\hat{f}$ is nondifferentiable on $\hat{N}(\hat{x}, \epsilon)$,
  it follows from the definition.

  Otherwise, if $\hat{f}$ is differentiable at all points on
  $\hat{N}(\hat{x}, \epsilon)$, then $|\hat{x}_i| > \epsilon$ for all $i$.
  Then because $\|\hat{y} - \hat{x}\|_1 \le \epsilon$,
  $\mathrm{sign}(\hat{y}_i) = \mathrm{sign}(\hat{x}_i)$ for all $i$,
  and $\nabla \hat{f}(\hat{y}) = \nabla \hat{f}(\hat{x})$
  and we have the Lemma statement with equality.
\end{proof}

\begin{Theorem}
  Consider any $\hat{x}, \hat{s}$ and $\alpha \in [0,1]$.
  Let $\hat{y} := \hat{x} + \alpha (\hat{s} - \hat{x})$.
  Then
  \begin{align}
    \min_{\hat{d} \in \hat{T}(\hat{x}, \epsilon)} \frac{1}{\alpha^2} \left( \hat{f}(\hat{y}) - \hat{f}(\hat{x}) - \langle \hat{y} - \hat{x}, \hat{d} \rangle \right) \le \frac{2}{\epsilon} \|\hat{s} - \hat{x}\|_1
  \end{align}
\end{Theorem}

\begin{proof} We verify two cases, \\

  \noindent \textbf{Case 1: $\hat{y} \in \hat{N}(\hat{x}, \epsilon)$}

  If we define the subgradient $\hat{d}$ elementwise by
  $\hat{d}_i = \mathrm{sign}(\hat{y}_i)$,
  then $\hat{d} \in \partial \hat{f}(\hat{y})$.
  Observe that $\langle \hat{y}, \hat{d} \rangle = \|y\|_1$.
  Then the left hand side of \eqref{eq:maxelt-Cf-distance}
  is bounded above by
  \begin{align}
    \hat{f}(\hat{y}) - \hat{f}(\hat{x}) - \langle \hat{y} - \hat{x}, \hat{d} \rangle
    &= \|\hat{y}\|_1 - \|\hat{x}\|_1 - \|\hat{y}\|_1 + \langle \hat{x}, \hat{d} \rangle\\ 
    &= \langle \hat{x}, \hat{d} \rangle - \|\hat{x}\|_1
  \end{align}
  And because $\|\hat{d}\|_\infty \le 1$, this quantity is nonpositive.
  We therefore only need to bound the second case, next.

  \noindent \textbf{Case 2: $\hat{y} \notin \hat{N}(\hat{x}, \epsilon)$}
  Here we have the lower bound on $\alpha$ that
  \begin{align}
    \alpha \ge \frac{\epsilon}{\| \hat{s} - \hat{x} \|_1}
  \end{align}
  and because $\|\hat{d}\|_\infty \le 1$,
  \begin{align}
     \hat{f}(\hat{y}) - \hat{f}(\hat{x}) - \langle \hat{y} - \hat{x}, \hat{d} \rangle
    &\le \|\hat{y}\|_1 - \|\hat{x}\|_1 + \|\hat{y} - \hat{x}\|_1 \\
    & = \|\hat{x} + \alpha (\hat{s} - \hat{x})\|_1 - \|\hat{x}\|_1 + \alpha \|\hat{s} - \hat{x}\|_1 \\
    \intertext{Using the triangle inequality on the first term,}
    &\le 2 \alpha \|\hat{s} - \hat{x}\|_1
  \end{align}
  The theorem statement then follows from this inequality and the bound on $\alpha$.
\end{proof}

So, much like in the piecewise linear case, we have:

\begin{Corollary}
  Given the problem of minimizing a piecewise linear function
  $f(x) = \|Ax + b\|_1$ over a bounded set $D$,
  the corresponding $C_f$ is bounded above by
  $ C_f(\epsilon) \le \frac{2}{\epsilon} \left( \mathrm{diam}_1(AD) \right)^2$.
  \label{cor:1norm-Cfbound}
\end{Corollary}

Note that in the case of multiway graph cuts,  $D$ consists of the Cartesian
product of $n$ simplices of dimension $d$.
We therefore expect the bound from Corollary \ref{cor:1norm-Cfbound}
will be $\Omega(n)$. The tractability, coreset (and hence the sublinearity)  discussion in this case follows directly from the $\ell_1$-regularized SVM case.

\subsection{1-median}
\label{1-median}
We can express the 1-median problem as an optimization over the simplex of
convex combinations of the input points.
Suppose we are given a (multi-)set of points
$\mathcal{P} = \{p_1, ..., p_n\} \subset \R^d$.
Let $A \in \R^{d \times n}$ be a matrix with $p_i$ in column $i$.
We can then write the 1-median problem as:
\begin{align}
    \min_{x \in \Delta_n} f(x) := \hat{f}(Ax) := \frac{1}{n} \sum_{i=1}^n \hat{f}_i(Ax) := \frac{1}{n} \sum_{i=1}^n \|Ax - p_i\|_2.
  \label{eq:median-obj-def}
\end{align}
The notation $\|\cdot\|$ denotes the $\ell_2$ norm.
Here, $f$ is nondifferentiable for any $x$ such that $Ax \in \mathcal{P}$,
and differentiable elsewhere.
We show that
$C_f(\epsilon) \le \frac{2}{\epsilon} \mathrm{diam}_2(A \Delta_n)^2$,
and extend this to a coreset result.

The subdifferentials are given by:
\begin{flalign}
  \partial f(x) &= A^T \partial \hat{f}(Ax) = \frac{1}{n} A^T \sum_{i=1}^n \partial \hat{f}_i(Ax) \\
  \partial \hat{f}_i(Ax) &= \begin{cases}
    \{\frac{Ax - p_i}{\|Ax - p_i\|}\} & \text{if } Ax \neq p_i \\
    \bar{\mathcal{B}}(0, 1) & \text{if } Ax = p_i
  \end{cases}
\end{flalign}
Note that the $\frac{1}{n}$ factor in \eqref{eq:median-obj-def}
is necessary to produce sensible approximation bounds.
If we duplicate each point in our input set,
this will double the value of the sum in the objective for any choice of
median.
Without the $\frac{1}{n}$ factor, any approximation bound expressed in terms of
the objective would necessarily be $\Omega(n)$.

The $k$-median problem has been considered previously in the coreset setting,
Har-Peled and Kushal \cite{har-peled2005-smaller-kmeans} give a coreset
construction by binning the points to derive a coreset of
size $O(k \epsilon^{-d} \log n)$.
For $k > 1$, the $k$-median problem is known to be nonconvex, so we here consider the
problem of calculating a single median.

For this problem, we choose neighborhoods as
\begin{align}
  N(x, \epsilon) = \{u \; | \; \|Au - Ax\| < \epsilon \text{ and } u - x \in \mathrm{im}(A^T) \}
  \label{eq:median-neighborhoods}
\end{align}
This meets the assumptions in the convergence bounds above,
as $x \in N(x, \epsilon)$,
and $N(x, \epsilon) \subseteq B(x, \|A^\dagger\| \epsilon)$.
$\|A^\dagger\|$ is the operator norm of the pseudoinverse.
The subdifferentials are described by:
\begin{align}
  \partial f(x) = \frac{1}{n} A^T \sum_{i=1}^n \partial \hat{f}_i(Ax), \quad
  \partial \hat{f}_i(Ax) = \begin{cases}
    \{\frac{Ax - p_i}{\|Ax - p_i\|}\} & \text{if } Ax \neq p_i \\
    \bar{\mathcal{B}}(0, 1) & \text{if } Ax = p_i
  \end{cases}
\end{align}

\subsection{Iteration bounds}
\label{sec:median-convergence-bounds}

In order to bound $C_f(\epsilon)$, we bound the quantity:
\begin{align}
    \frac{1}{n \alpha^2}
    \min_{u \in N(x, \epsilon)} & \left(
    \sum_{Au \neq p_i} \left( \|Ay - p_i\| - \|Ax - p_i\|
    \right. \right. \\ &\left.\left.- \left\langle Ay - Ax, \frac{Au - p_i}{\|Au - p_i\|} \right \rangle \right) \right. \\
    & \left. + E(u) (\|Ay - Au\| - \|Ax - Au\| - \|Ay - Ax\|)
    \right)
\end{align}
for any $x,s$ in the simplex and $y = x + \alpha (s - x)$ for $\alpha \in [0,1]$.
$E(u)$ is the number of points such that $Au = p_i$.
This is equivalent to the contents of the supremum in $C_f(\epsilon)$,
where we have substituted the analytic solution for the minimum
element of $\partial f(u)$.
By the triangle inequality, for any choice of $u$,
\begin{align}
  \|Ax - Au\| + \|Ay - Ax\| &\ge \|(Ax - Au ) + (Ay - Ax)\| \\ 
  &= \|Ay - Au\|
\end{align}
Since the nondifferentiable terms are nonpositive,
we only need to bound the differentiable terms.
First, assume that $y - x \in \mathrm{im}(A^T)$.
This is without loss of generality, as we can replace $s$ with its projection
$s' \in x + \mathrm{im}(A^T)$ and $y$ with $y' = x + \alpha(s' - x)$.
Because $s' - s \in \mathrm{ker}(A)$, this yields $Ay' = Ay$,
and the quantity we seek to bound will be identical.\\

\noindent \textbf{Case 1: $y \in N(x, \epsilon)$.}
If $Au = Ay$, the $i^\text{th}$ differentiable term is
\begin{align}
  & \|Ay - p_i\| - \|Ax - p_i\| - \left\langle Ay - Ax, \frac{Ay - p_i}{\|Ay - p_i\|} \right\rangle \\
  =& \frac{1}{\|Ay - p_i\|} \left(
    \langle Ay - p_i, Ay - p_i \rangle
    - \|Ax - p_i\| \|Ay - p_i\| \right.\\&\left.- \langle Ay - Ax, Ay - p_i \rangle
    \right) \\
  =& \frac{1}{\|Ay - p_i\|} \left(
    \langle Ax - p_i, Ay - p_i \rangle
    - \|Ax - p_i\| \|Ay - p_i\|
  \right),
\end{align}
which by the Cauchy-Schwartz inequality is nonpositive.
Therefore, in order to provide an upper bound on $C_f(\epsilon)$
we need only consider the second case below.

\noindent \textbf{Case 2: $y \notin N(x, \epsilon)$.}
First observe that in this case we have a lower bound on $\alpha$.
Specifically, because $y$ lies outside the neighborhood
around $x$ defined by \eqref{eq:median-neighborhoods}
(and in the set $x + \mathrm{im}(A^T)$ by assumption), 
$\|Ay - Ax\| = \alpha \|As - Ax\| \ge \epsilon$,
\begin{align}
  \alpha \ge \frac{\epsilon}{\|As - Ax\|} \ge \frac{\epsilon}{\mathrm{diam}(A \Delta_n)}
  \label{eq:median-alpha-bound}
\end{align}
Consider in the following \emph{any} choice of $u$ from the neighborhood
around $x$.
Plugging in $y = x + \alpha (s - x)$ to the differentiable terms:
\begin{align}
  & \frac{\sum_{Au \neq p_i} 
\|Ay - p_i\| - \|Ax - p_i\| - \left\langle Ay - Ax, \frac{Ax - p_i}{\|Ax - p_i\|} \right\rangle}{n \alpha^2}
    \\
  =&  \frac{\splitfrac{\sum_{Au \neq p_i}
\|Ax - p_i + \alpha (As - Ax)\|} {- \|Ax - p_i\| 
  - \alpha \left\langle As - Ax,  \frac{Au - p_i}{\|Au - p_i\|} \right\rangle}}{n \alpha^2} 
  \intertext{Using the triangle and Cauchy-Schwartz inequalities,}
  \le& \frac{1}{n \alpha^2} \sum_{Au \neq p_i} 2 \alpha \|As - Ax\|
  \le \frac{2 \|As - Ax\|}{\alpha}  \le \frac{2 \mathrm{diam}(A\Delta_n)}{\alpha}.
  \label{eq:median-differentiable-bound}
\end{align}
The combination of \eqref{eq:median-alpha-bound} and
\eqref{eq:median-differentiable-bound},
taken over all $x,s$ and all $\alpha$ that fall in this case,
gives:
\begin{align}
  C_f(\epsilon) \le \frac{2}{\epsilon} \mathrm{diam}(A \Delta_n)^2
\end{align}

Next, observe that if $\epsilon$ is small, as will be the case throughout
the optimization if we choose $\epsilon_k \in \Theta(\sqrt{\alpha_k})$ with
a small constant term,
then with high probability the neighborhood around the iterates for any $k$
that are not near the solution will be differentiable.
This will yield $T(x, \epsilon) = \{\nabla f(x)\}$
and the subproblems will be a linear program over the simplex,
which will only introduce no more than one nonzero as in
\cite{clarkson2008-frank-wolfe} which indeed means that the procedure is sublinear.

\subsection{Sparse PCA with constraints}
\label{sec:spca}

Sparse PCA is an important problem that is widely used for feature extraction and learning.  
In this problem we consider the formulation given in \cite{Grbovic2012-SparsePCAcons}. The objective of sparse PCA is to decompose a covariance matrix $S$ into near orthogonal  principal components $[u_1,...,u_k]$, while constraining the number of nonzero components of each $u_k$ to a required constant $r\in\mathbb{N}$. Hence the problem of maximizing the variance of component $u_k$ with the cardinality constraint can be written as,
\begin{align}
\max_u  \quad u^TSu\quad \text{s.t.}  \quad ||u||_2 =1, ~~\text{card}(u)\leq r
\end{align}
Writing down the convex relaxation of the problem after setting $U=uu^T$, we get the following optimization problem,
\begin{align}
\max_U  \quad \text{Tr}(SU)\quad\text{s.t.}  \quad \text{Tr}(U)=1, ~~1^T|U|1\leq r, ~~U =U^T,~~U \succeq 0, ~~ U\in\mathbb{R}^{n\times n}
\end{align}
The last inequality is the standard Linear Matrix Inequality that requires $U$ to be positive semidefinite. But note that we have exponential number of linear constraints in the form of $1^T|U|1\leq r$ constraint. Hence in  general this problem might practically take very long to solve unless one uses a specialized interior point method that exploits the structure of the constraints. Hence these constraints are taken to the objective with a penalty $\rho$ associated with it as,
\begin{align}
\max_U  \quad \text{Tr}(SU)- ||\rho\text{vec}(U)||_1 ~~
\text{s.t.}  \quad \text{Tr}(U)=1,  ~~U =U^T,~~U \succeq 0, ~~ U\in\mathbb{R}^{n\times n}
\end{align}
where vec$(U)$ denotes vectorized form of $U$.  To this formulation, \cite{Grbovic2012-SparsePCAcons} suggests adding a class of feature grouping constraints that is motivated by the maintenance scheduling problem \cite{cui2008maintenance}. Let us denote by $\textbf{1}$, a reliability vector such that  $\textbf{1}_i\in\left[0,1\right)$ is a probability that sensor $i$ will need maintenance during a certain time period. Then the \textit{reliability} matrix $L$ is constructed by setting each column to $\textbf{1}$ and then setting $L_{ii}=0~\forall ~1\leq i\leq n$. Using similar techniques we end up with the following formulation,
\begin{align}
\max_U & \quad \text{Tr}(SU)-||\rho\text{vec}(U)||_1 - ||\text{vec}(\eta U\circ \log(1-L))||_1\\
\text{s.t.} & \quad \text{Tr}(U)=1,  ~~U =U^T,~~U \succeq 0, ~~ U\in\mathbb{R}^{n\times n}
\end{align}
where $\eta$ is the penalty that controls the reliability of the component and $\circ$ is the standard Hadamard product or elementwise product. Let $\vec{U}$ denote the vectorized form of $U$ for simplicity in notation and similarly $\text{vec}(\log(1-L)) = \vec{l}$. Note that $\vec{l},\vec{U}\in\mathbb{R}^{n^2\times 1}$. Then the problem can be written as,
\begin{align}
\max_U & \quad \text{Tr}(SU)-\sum_{i=1}^{n^2} \left( \rho + \eta\vec{l}_i\right) |\vec{U}_i|\\
\text{s.t.} & \quad \text{Tr}(U)=1,  ~~U =U^T,~~U \succeq 0, ~~ U\in\mathbb{R}^{n\times n}
\end{align}
Noting that Tr$(SU) = \sum_{i=1}^n\sum_{j=1}^nS_{ij}U_{ij}=
\vec{S}^T\vec{U}=\sum_{i=1}^{n^2}\vec{S}_i\vec{U}_i$, we can write the objective function as a supremum of $2^{n^2}$ affine functions. Also using the min-max theorem we convert this to a minimization problem. Hence the problem becomes,
\begin{align}
\min_{{U}} & \quad \max_i \left(a_i^T\vec{U} \right) \\
\text{s.t.} & \quad \text{Tr}(U)=1,  ~~U =U^T,~~U \succeq 0, ~~ U\in\mathbb{R}^{n\times n}
\end{align}
Let $S:=\{U:\text{Tr}(U) = 1,~~U=U^T,~~U\succeq 0,~~U\in\mathbb{R}^{n\times n}\}$. 
\begin{Lemma}
The feasible set of sparse PCA with constraints problem, $S$ is bounded.
\end{Lemma}
\begin{proof}
We know that $U\succeq 0,~~U=U^T \iff \lambda_i\geq 0$ where $\lambda_i$s are the eigenvalues of  $U$ and $\text{Tr}(U) = 1\iff \sum_i\lambda_i = 1$. Let $S_\lambda:= \{\lambda : \sum_i\lambda_i=1,~~\lambda\geq 0,~~\lambda\in\mathbb{R}^{n}\}$. Hence $S$ is bounded if and only if $S_\lambda$ is bounded. But $S_\lambda$ is trivially bounded since it is the simplex.
\end{proof}
Let $A\in\mathbb{R}^{2^{n^2}\times n^2}$ be the matrix containing all the $a_i$s stacked on top of each other. Hence our optimization problem can be written as,
\begin{align}
\min_{{U}}  \quad \max_i \left(A\vec{U}\right)_i \quad
\text{s.t.}  \quad U \in S
\end{align}
Now we define the neighborhoods in a similar way as Section \ref{sec:piecewise-linear},
\begin{align}
\hat{N}\left( \hat{\vec{U}},\epsilon\right)&= \{\hat{\vec{X}}\; | \; \|\hat{\vec{U}}-\hat{\vec{X}}\|_\infty\leq \epsilon\};\\
N\left(\vec{U},\epsilon\right)&=\{\vec{X}\; | \;\|A\vec{X}-A\vec{U}\|_\infty\leq \epsilon \text{ and } \vec{X}-\vec{U}\in \text{im}\left(A^T\right)\}.
\end{align}
Now define 
\begin{align}
V\left(\hat{\vec{U}},\epsilon\right)=\left\{e_i \; | \;\hat{\vec{U}}_i\geq \left( \max_j \hat{\vec{U}}_j\right) - 2\epsilon\right\}
\end{align}
With these constructions, similar results as in Section \ref{sec:piecewise-linear} are obtained. Note that unlike the other cases, the subproblems here are not as easy to solve and 
we might require advanced numerical optimization solvers to solve the subproblems efficiently in practice.

\subsection{Generic Subproblems}
As we saw in the above sections, specialized algorithms can be used to solve the subproblems that are induced by \eqref{eq:alg-s-selection} for many important applications. But occasionally, we might encounter problems that do not possess an inherent structure that can be easily exploited; for these cases, we now provide a generic method (or algorithm) for solving the subproblems. Specifically, when the optimization in \eqref{eq:alg-s-selection} is over a union of convex sets $C=\cup_{i=1}^{\mathfrak{n}} C_i$ where $C_i$ is a convex set  $\forall~i=1,...,\mathfrak{n}$, commonly encountered in practice. 
Recall that the subdifferential of a proper convex function is a compact set for all the points in the interior of the domain \cite{robinson2012-convex}. Assuming that we have access to the zeroth and first order oracles, $C$ can be explicitly described by the inequalities $g_{ij}(x)\leq 0$ where $g_{ij}$ are convex functions $\forall~i,j$ and the total set of constraints has finite cardinality. It is easy to show that the problem of optimizing a convex function $g$ over $C$ can be formulated as the following convex optimization problem,
\begin{align}
    &\min_{x} \quad  g(x) \\
    &\text{s.t.} \quad
    s_ig_{ij}\left(z_i/s_i\right)\leq 0,\quad \forall~i,j\\
    &\quad \quad 1^Ts=1,s\geq 0, x=\sum_{i=1}^{\mathfrak{n}}z_i
  \label{union}
\end{align}
There are two advantages of using this reformulation instead of solving $\mathfrak{n}$ different convex problems: (1) in certain cases, it might be easy to design algorithms to get approximate solutions 
much more efficiently \cite{garber2011approximating} and (2) standard solvers are actively being developed to handle perspective reformulations as described in \cite{gunluk2012perspective}. 
From this discussion, we can conclude that the subproblems can be solved in polynomial time.

\section{Experiments}
\label{sec:exps}
In this section, we present experiments to assess our convergence
rate numerically and provide some practical intuition for the bounds obtained in the sections above. 
We provide the running time whenever it is significant (more than a few seconds). 
To keep the extraneous effects of specialized libraries and the number of processors negligible, 
all our subproblems were solved  using a generic Linear Programming solver on a single core. 
The results show that the algorithm presented here is not only practical but can also be made competitive using simple alternative standard techniques 
to solve the subproblems.

\subsection{Support Vector Machines with $\ell_1$ regularization}
We first demonstrate results of our algorithm for $\ell_1$-regularized SVM on a
collection of four test datasets provided by the authors of
\cite{yuan2010-l1svm-optim} (see Figure \ref{fig:svm-convergence-2}).
\begin{figure}[!ht] \centering 
\includegraphics[width=\linewidth]{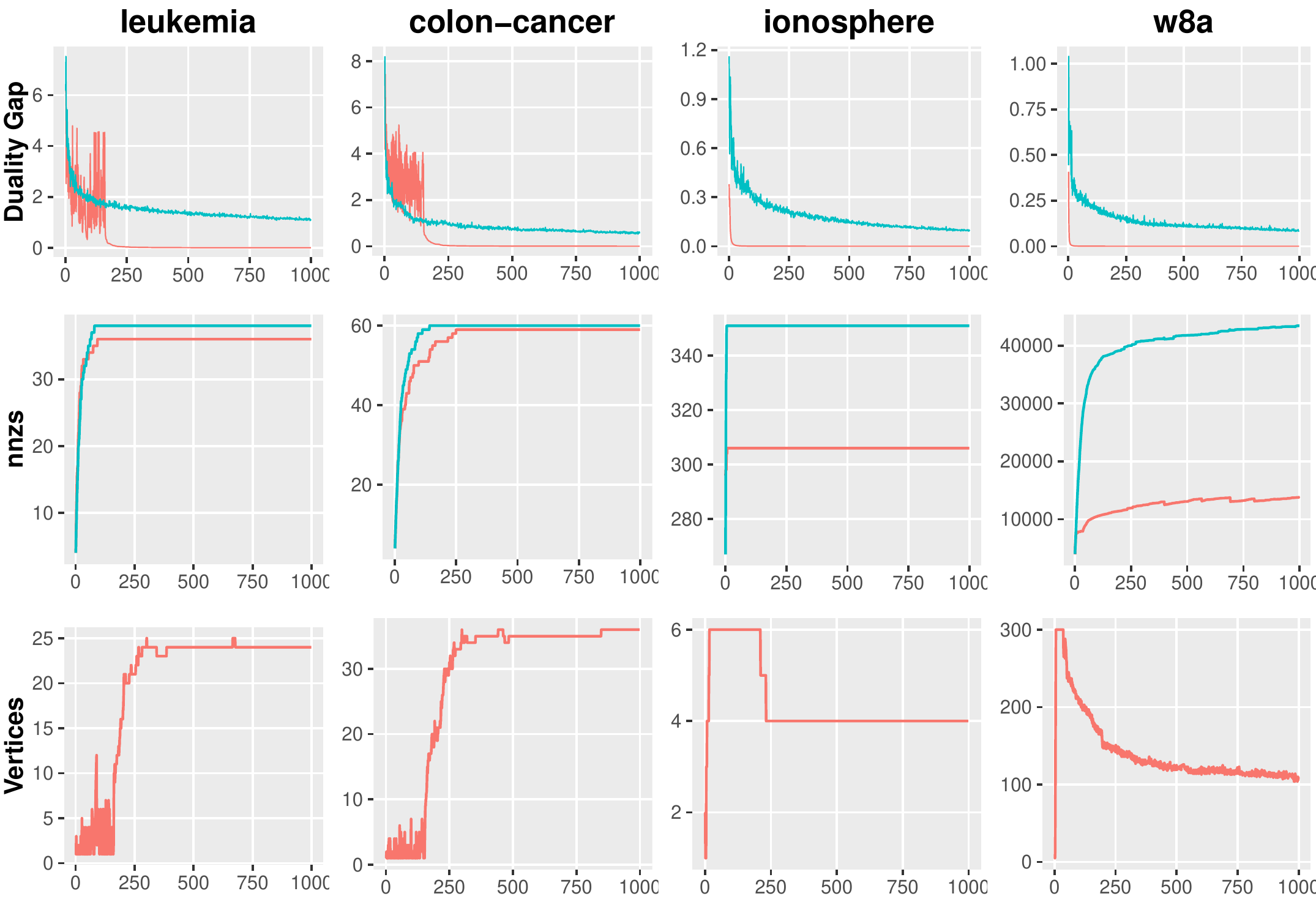}
\caption{ \label{fig:svm-convergence-2}  Convergence results for 
    Alg. \ref{alg:nonsmooth-stepsizek} (red) and
    randomized smoothing of \cite{lan2012-linopt-oracle} (blue) on the
    SVM datasets\eqref{eq:bennet-l1svm-dual}.}
\end{figure}

For this evaluation, we also compared our method with the randomized smoothing method of \cite{lan2012-linopt-oracle} shown in blue in Fig. \ref{fig:svm-convergence-2}.
The top row of plots in Figure \ref{fig:svm-convergence-2} ($y-$axis showing the duality gap) suggest that on all four datasets (leukemia, colon-cancer, ionosphere, w8a), the convergence of the baseline 
is slower as a function of the number of iterations. The second and the third row show that the coreset size and the vertices in $ |V(Ax,\epsilon)|$ increase {\em at most linearly} with the number of iterations as predicted in Section \eqref{sec:l1-svm}. 
We see that the number of vertices in $ |V(Ax,\epsilon)|$ remains {\em very small} throughout our experiments. 
While the convergence results shown for \cite{lan2012-linopt-oracle} are, in expectation, of
the same order as the ones we show here,
the bound is quite loose for these problems and we find that the practical convergence
rate can be slow. 
This is especially important as the computation time of each iteration of
that method increases with the iteration count, as more samples of the
gradient are drawn in later iterations (unlike our algorithm). 
Moreover, Figure \ref{fig:ls} shows that if we use a bisection line search instead of the step size policy suggested in the proof of Theorem \ref{thm:convergence}, 
our algorithm actually converges (measured using the duality gap) in fewer than 40 iterations 
making it practically fast while also keeping our theoretical results intact. 

\begin{figure}[!ht] \centering 
	\includegraphics[width=\linewidth]{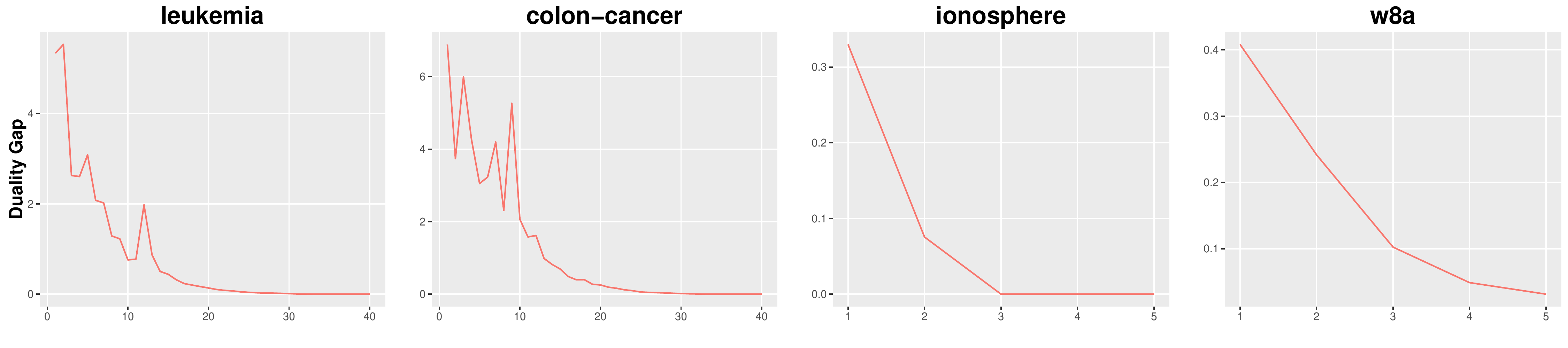}
	\caption{\label{fig:ls}   Convergence results for 
		Alg. \ref{alg:nonsmooth-stepsizek} using bisection line search.}
\end{figure}

\subsection{The 1-median problem}
We tested a simple Frank Wolfe based solver for the 1-median problem on synthetic data (similar to existing works for this problem). 
A set of points of size $n$ is sampled from a multivariate random Normal distribution.
To keep the presentation succinct, we discuss those parts of this experiment that demonstrate additional useful properties of 
the algorithm beyond those described for $\ell_1$-SVM. 
The left plot in the top row of Figure \ref{fig:median-n}  shows how the size of the coreset varies as
we increase $n$ whereas the right plot shows the number of iterations required for convergence. Importantly, the size of the coreset is \emph{independent} of $n$,
as the problem we are solving is fundamentally no harder. The second row shows the computational time taken: left plot shows the total time taken whereas the right plot shows the time taken per iteration. We see that both of these scale sublinearly with the number of points. The plots agree with the theoretical results from earlier sections concerning the size of coreset and number of iterations required to obtain 
a solution.

\subsection{Multiway Graph Cuts}
In order to demonstrate that our algorithm is scalable to large datasets and hence it is practically applicable, we tested our algorithm on the \emph{dblife} dataset \cite{niu2011-hogwild} which consists of 9168 labels. 
Therefore, $D$ is the product of $n$ 9168-dimensional simplices. Note that this problem consist of \emph{approximately 20 million decision variables} and has been recently 
tackled via distributed schemes 
implemented on a cluster \cite{niu2011-hogwild}. Due to the large of number of labels, this dataset is ideal for evaluating our algorithm as the analysis implies that the number of labels chosen must \emph{increase linearly with the number of iterations}. This means that we get a sparse labeling scheme, i.e., the {\em coreset here is the subset of the total number of labels}. 
We evaluated this behavior by using increasing subsets of the total number of labels. We find that even when the total number of labels is small; this trend becomes more prominent as 
progressively larger sets of labels are used, 
shown in Figure \ref{fig:median-n} (row 3).

\begin{figure*}[!hbt]
  \vspace{-1em}
  \centering
    \subfloat{\includegraphics[width=0.35\linewidth]{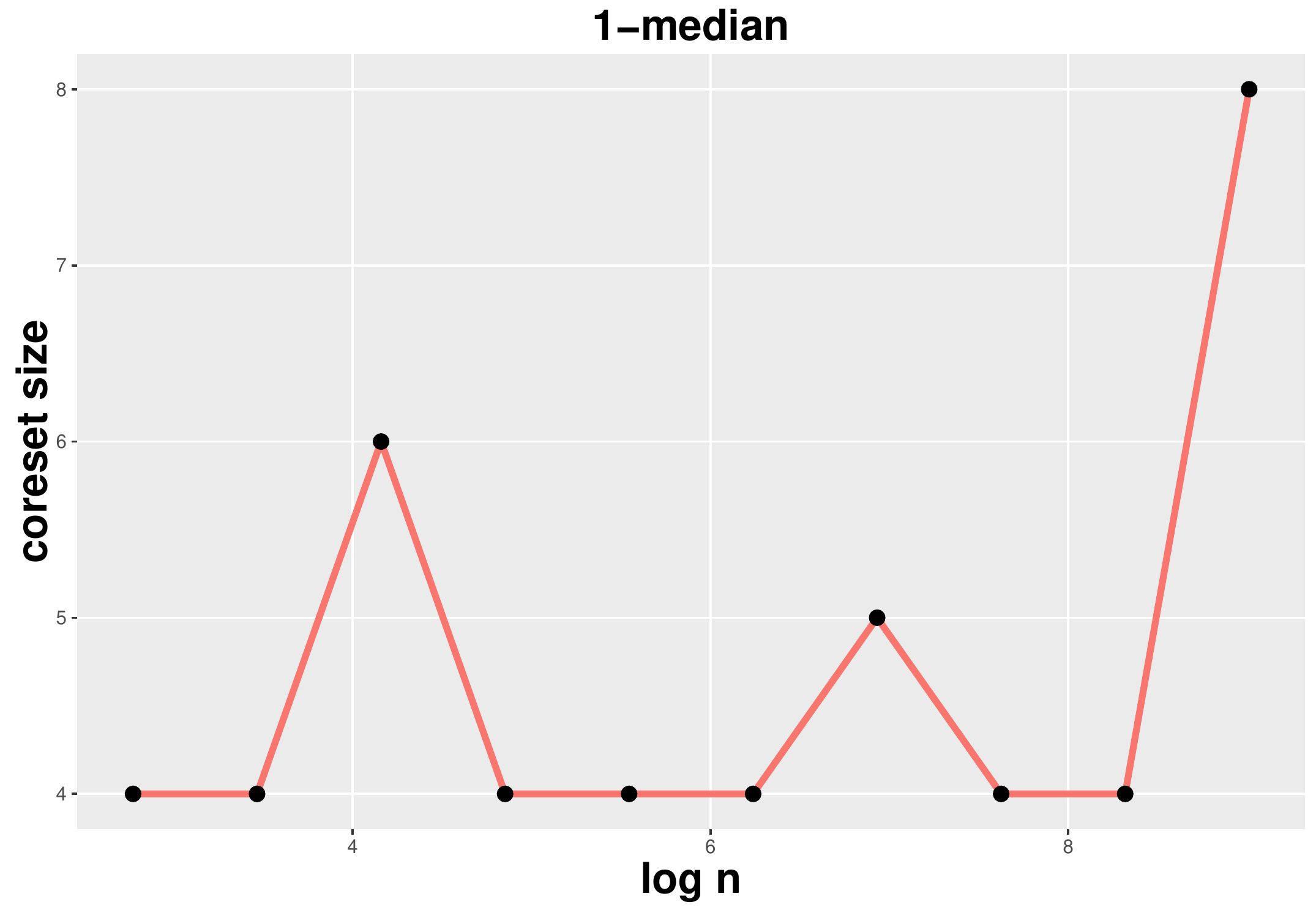}} 
    \subfloat{\includegraphics[width=0.35\linewidth]{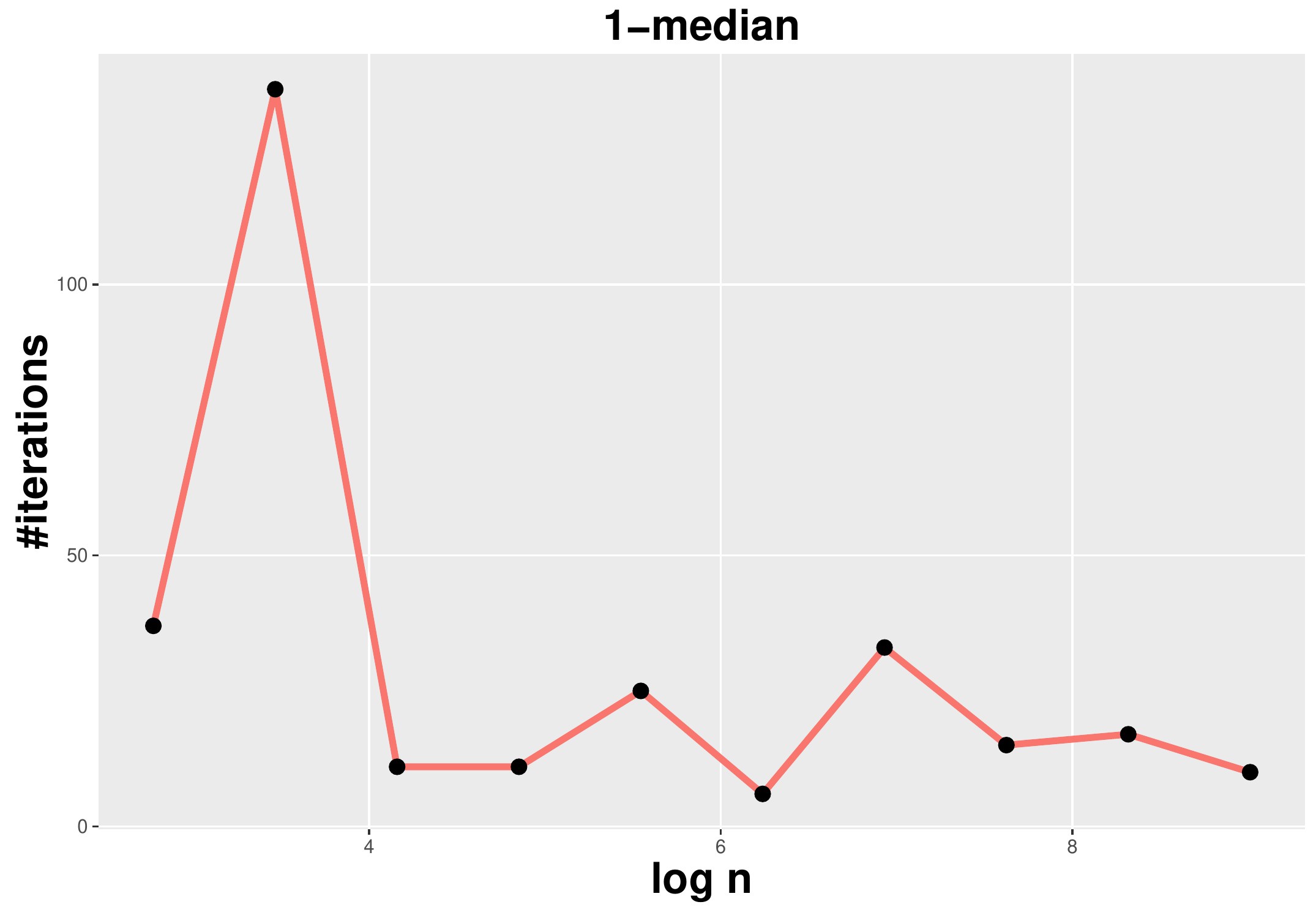}}
      ~~  \\
      \subfloat{\includegraphics[width=0.35\linewidth]{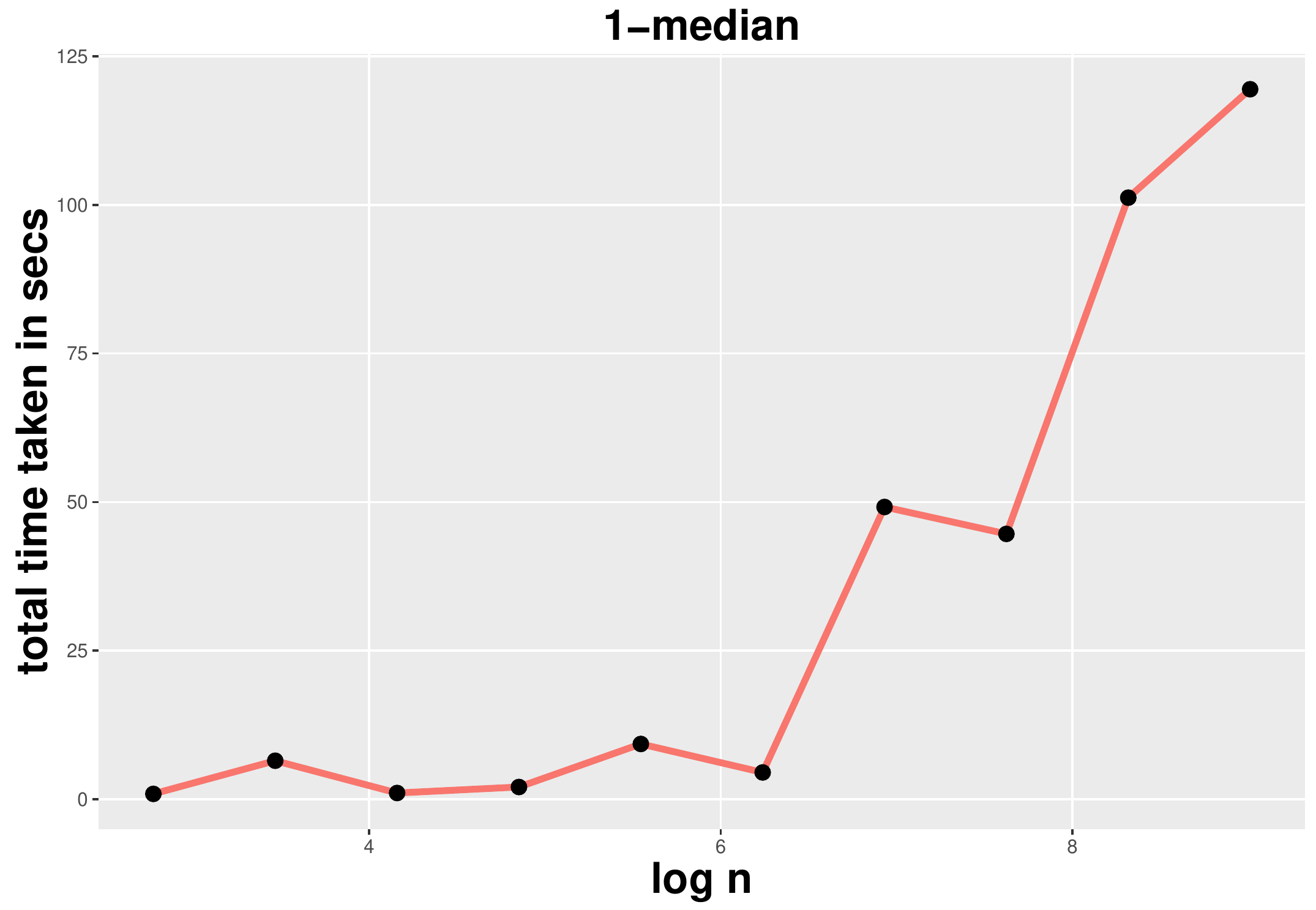}} 
    \subfloat{\includegraphics[width=0.35\linewidth]{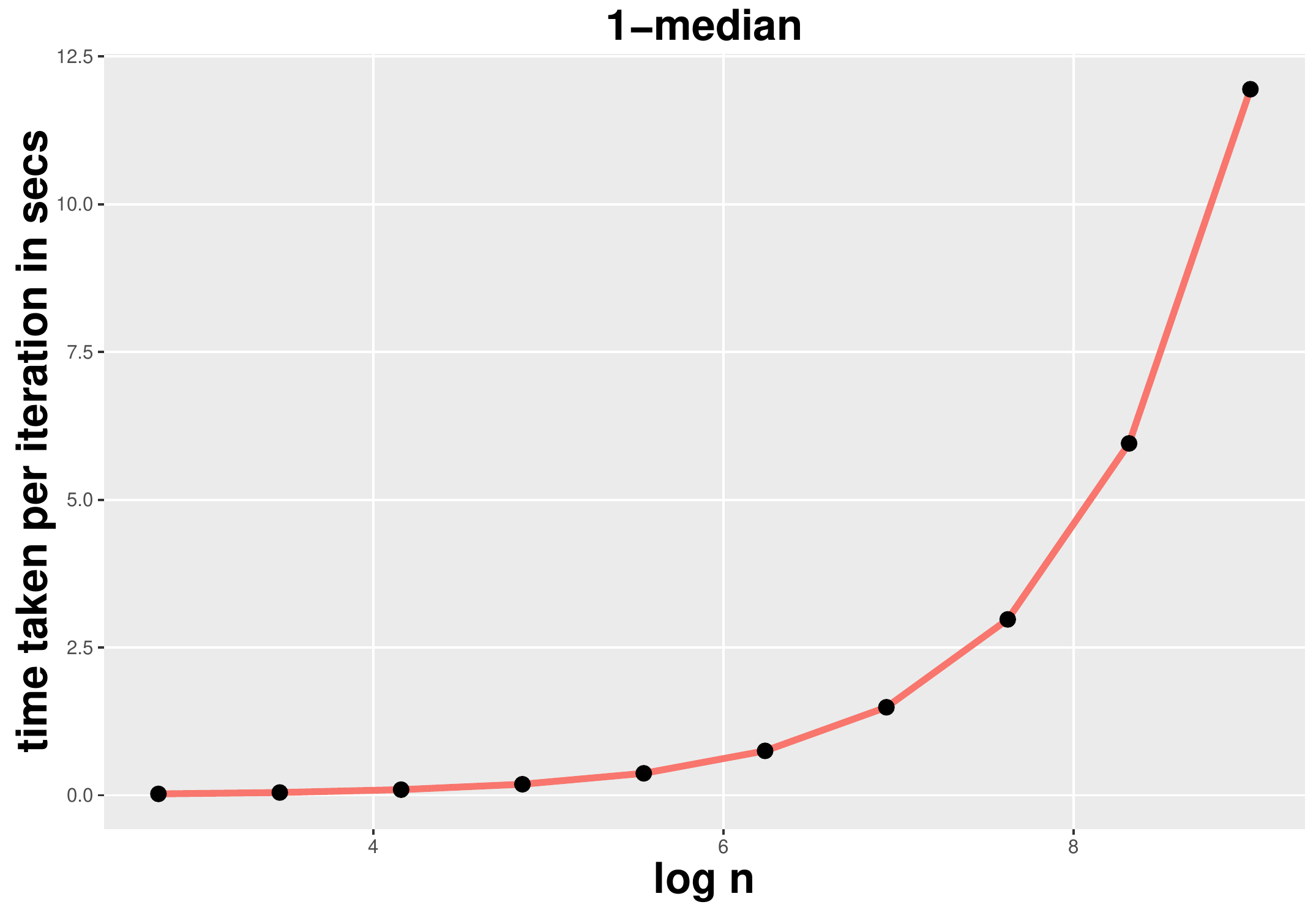}}
      ~~  \\
    \subfloat{\includegraphics[width=0.35\linewidth]{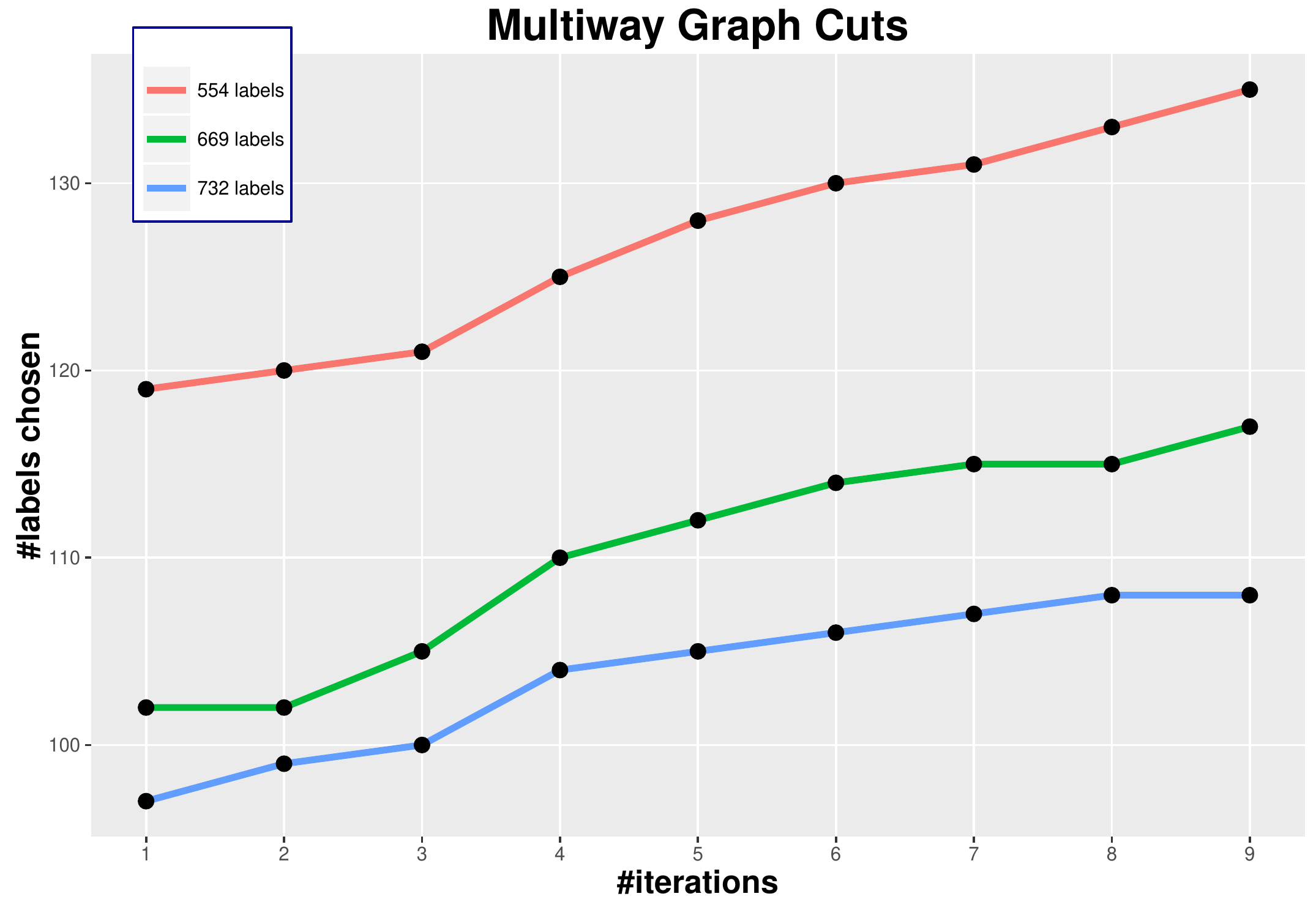}}
  \vspace{-8pt}
  \caption{ \label{fig:median-n}  (Top row)  Plotting the size of the coreset (left) and number of iterations (right) for 1-median as we vary the size $n$ of a randomly sampled set of points. We iterate until achieving a primal-dual gap of $10^{-6}$; $2^{\text{nd}}$ row  shows the total time taken (left) and time taken per iteration (right); $3^{\text{rd}}$ row shows results for Graph cuts problem on dblife. Each LP subproblem took approximately 60 seconds using CPLEX solver. }
\vspace{-1em}

\end{figure*}

\begin{table}[!h]
\begin{tabular}{|l|p{0.35\textwidth}|l|}
\hline
Problem & $C_f$ & Coreset\\                 
 \hline
$\ell_1$-SVM & $\frac{2}{\epsilon}(\text{diam}_{\infty}(A^+\Delta_m)+\text{diam}_{\infty}(A^{-}\Delta_m))$ & Support vectors \\
1-median     & $\frac{2}{\epsilon}\text{diam}_2(A\Delta_n)^2$                                                                   & Points \\
Graph cuts   & $\frac{2}{\epsilon}(\text{diam}_1(AD))^2$ & Labels \\
Balanced Development & $\frac{2}{\epsilon}( \text{diam}_1(A\Delta_m))^2$ &  Attributes\\
Sparse PCA with constraints &  $\frac{2}{\epsilon}( \text{diam}_{\infty}(AS))^2$ & Features\\
\hline
\end{tabular}
\caption{\label{tab:summary-tab} A summary of the relevant quantities derived for each problem described in the previous sections.
}
\end{table}

\subsection{ Interesting empirical behavior and its potential relation to existing work}
The experiments highlight two important practical aspects of the proposed algorithm. \\
\emph{A) How often do we encounter poor subproblems?}  The number of vertices in $|V(Ax, \epsilon)|$ given by our analysis seems to be an overestimate than what is practically observed. Assuming that we initialize at a point whose neighborhood does not contain any nonsmooth points, using the fact that the set of nondifferentiable points of a convex function on a compact set is of measure zero, we can show that the next iterate of the algorithm also does not contain any nonsmooth points in its neighborhood with high probability. But extending this argument to the next iterate becomes problematic. It seems like a Lovasz Local Lemma type result \cite{kolmogorov2015commutativity} will be able to shed some light on this behavior but the independence assumption is inapplicable, so one has to introduce randomness in the algorithm. 
But a separate sparsity analysis might have to be done to obtain coreset results.\\
\emph{B) Why does our algorithm converge so fast?} The empirical behavior of certain algorithms have been strongly associated with the \emph{manifold identification property} in recent works, see \cite{lee2012manifold}. Our algorithm when using a bisection line search strategy as shown in fig \ref{fig:median-n}, converges within a few iterations and is faster than the state of the art solvers for that problem. 
It will be interesting to see if our algorithm satisfies some variation or extension of the 
properties that are described in \cite{lee2012manifold}, but one issue is that the algorithm presented in \cite{lee2012manifold} 
has randomization incorporated at each iteration making it hard to adapt that analysis for our algorithm.\\
\emph{C) Tightness of convergence:} Let us consider the case of $\ell_1$-SVM. Assume that the positive and negative labeled instances be generated by a Gaussian distribution with mean $0$ and $\mu$ respectively and variance $\sigma I$. Then we have that, 
\begin{align}
\mathbb{E}\left(\text{diam}_{\infty}(A^+\Delta_m)+\text{diam}_{\infty}(A^-\Delta_m)\right)\leq \sigma\log m + \mu + \sigma \log m = 2\sigma\log m + \mu
\end{align}
using the properties of maximum of sub-gaussians. Here the expectation is taken over the randomness in the data: $A^{+},A^{-}$. This shows that the dependence of $C_f$ is logarithmic in the dimensions. Similarly we will get  $\sqrt{n}, n\log n$ for $\ell_2,\ell_1$ norm diameters respectively showing that the $\mathbb{E}C_f$ is near linear in the dimensions. Now we plug this estimate in the convergence result in theorem \eqref{thm:convergence}, 
\begin{align}
\mathbb{E}\left( f( x_k ) - f(x_*)\right) \le \mathbb{E}\left(\frac{2^{5/2} L + 2^{3/2} D_f}{\sqrt{k + 2}}\right) = 2^{5/2}\left(\frac{L+\mathbb{E}D_f}{\sqrt{k+2}}\right)=O\left(\frac{L+ \log m}{\sqrt{k}}\right)
\end{align} 
Hence we can see that by choosing an appropriate neighborhood depending on the application, we can get faster convergence. This aspect is either not present in other Frank Wolfe algorithms that solve nonsmooth problems or involve  a looser dependence on the constants involved for example, \cite{hazan2012projection} achieve a $O\left(\frac{L m}{\sqrt{k}}\right)$ convergence compared to our $O\left(\frac{L+\log^2 m}{\sqrt{k}}\right)$ here.
\section{Discussion}

We have presented novel coreset bounds and optimization schemes for nondifferentiable 
problems encountered in statistical machine learning and computer vision.
The  algorithm calculates
sparse approximate solutions to the corresponding optimization problems deterministically in a
number of iterations independent of the size of the input problem,
depending only on the approximation factor and the sum of the Lipschitz constant
and a nonlinearity term.
The central result in Theorem \ref{thm:convergence}  applies to
any problem of the very general form in \eqref{eq:primal-problem} with
mild conditions, potentially suggesting a number of other applications beyond those considered here. 
Though a general condition to characterize {\em all} cases where the internal subproblem is efficiently solvable may not be available, we show a broad/useful class that \emph{is} efficiently solvable.  

Finally, we point out a few technical properties that differentiate our method from existing FW type methods for nonsmooth problems \cite{argyriou2014hybrid, pierucci2014smoothing}.

{\em First}, many methods rely on smoothing the objective function by a \emph{prox} function that needs to be designed piecemeal for specific problems. 
While examples exist where these functions can be readily derived, to our knowledge there are no standard recipes for deriving proximal functions in general. 
{\em Second}, several existing methods assume that the proximal iteration can be solved efficiently. However, it is known that in general, the worst case complexity of a single proximal iteration 
is the same as solving the original optimization problem \cite{parikh2014proximal}. 
{\em Third}, it is not yet clear how coreset results can be derived for many existing methods, or if it is even possible. For example, (especially) in the \emph{very large scale} setting, 
coreset results enable practical applications where one can store a subset of the (training) dataset and still be able to perform nearly as well (on the test data). The basic expectation 
is that more data should \emph{not} always make a problem computationally harder. An exciting implication behind coreset results in \citep{clarkson2008-frank-wolfe} is that this can be avoided in certain cases. 
But \citep{clarkson2008-frank-wolfe} assumes that $f$ is smooth: an artifact of the optimization rather than a requirement of coresets per se --- we showed that this assumption is not necessary at all. 
In closing, 
while the algorithm may not be the defacto off-the-shelf option for {\em all} nonsmooth problems, 
for many problems it offers a very competitive (generally applicable) alternative, 
and in certain cases, the theory nicely translates into significant practical benefits as well.

\section{Acknowledgments}
The authors are supported by NSF CAREER RI \#1252725,  NSF CCF \#1320755, 
and \href{http://cpcp.wisc.edu/}{UW CPCP} (U54 AI117924). 
We thank Stephen J. Wright, Shuchi Chawla, Satyen Kale and Kenneth L. Clarkson
for comments and suggestions. 

\newpage
\bibliographystyle{plain}
\bibliography{nonsmooth_fw_jmlr}
\end{document}